\documentclass[10pt, a4paper, oneside]{article}

\usepackage[scaled]{helvet} 
\usepackage{courier} 
\usepackage{eulervm} 
\usepackage{lscape}
\normalfont
\usepackage[english]{babel}
\usepackage[utf8]{inputenc}
\usepackage{indentfirst}				
\usepackage{booktabs}				
\usepackage{multirow}					
\usepackage{tabularx}					
\usepackage{graphicx}					
\usepackage{caption}					
\usepackage{amsmath,amssymb,amsthm,mathrsfs,mathtools}	
\usepackage{bm,faktor,braket,stmaryrd,esint}
\usepackage[english]{varioref}			
\usepackage{mparhack,relsize}			
\usepackage[dvipsnames]{xcolor}			
\usepackage{hyperref}					
\usepackage{bookmark}					
\usepackage{fullpage}
\usepackage{tcolorbox}
\usepackage[export]{adjustbox}
\usepackage{rotating}
\usepackage{pbox}
\usepackage{float}
\usepackage[capitalise,noabbrev]{cleveref}
	\crefname{equation}{equation}{equations}
\newsavebox{\mytable}
\usepackage{tikz}
	\usetikzlibrary{positioning, decorations.markings}
	
	\definecolor{lightyellow}{rgb}{1,1,0.878}
	\definecolor{orange}{rgb}{1,0.647,0}

\usepackage[textsize=footnotesize]{todonotes}
\usepackage{tikz-cd}
\usetikzlibrary{arrows,decorations.pathreplacing,decorations.markings,calc}

\definecolor{linkred}{rgb}{0.75,0,0}
\definecolor{linkblue}{rgb}{0,0,0.75} 
\definecolor{darkblue}{RGB}{0, 0, 139}
\definecolor{darkred}{RGB}{139,0,0}
\definecolor{lightblue}{RGB}{179, 209, 255}
\definecolor{lightred}{rgb}{1,0.85,0.85}

\UseRawInputEncoding
\usetikzlibrary{arrows.meta,patterns}
\usetikzlibrary{ipe} 

\tikzstyle{ipe stylesheet} = [
  ipe import,
  even odd rule,
  line join=round,
  line cap=butt,
  ipe pen normal/.style={line width=0.4},
  ipe pen heavier/.style={line width=0.8},
  ipe pen fat/.style={line width=1.2},
  ipe pen ultrafat/.style={line width=2},
  ipe pen normal,
  ipe mark normal/.style={ipe mark scale=3},
  ipe mark large/.style={ipe mark scale=5},
  ipe mark small/.style={ipe mark scale=2},
  ipe mark tiny/.style={ipe mark scale=1.1},
  ipe mark normal,
  /pgf/arrow keys/.cd,
  ipe arrow normal/.style={scale=7},
  ipe arrow large/.style={scale=10},
  ipe arrow small/.style={scale=5},
  ipe arrow tiny/.style={scale=3},
  ipe arrow normal,
  /tikz/.cd,
  ipe arrows, 
  <->/.tip = ipe normal,
  ipe dash normal/.style={dash pattern=},
  ipe dash dashed/.style={dash pattern=on 4bp off 4bp},
  ipe dash dotted/.style={dash pattern=on 1bp off 3bp},
  ipe dash dash dotted/.style={dash pattern=on 4bp off 2bp on 1bp off 2bp},
  ipe dash dash dot dotted/.style={dash pattern=on 4bp off 2bp on 1bp off 2bp on 1bp off 2bp},
  ipe dash normal,
  ipe node/.append style={font=\normalsize},
  ipe stretch normal/.style={ipe node stretch=1},
  ipe stretch normal,
  ipe opacity 10/.style={opacity=0.1},
  ipe opacity 30/.style={opacity=0.3},
  ipe opacity 50/.style={opacity=0.5},
  ipe opacity 75/.style={opacity=0.75},
  ipe opacity opaque/.style={opacity=1},
  ipe opacity opaque,
]

\usepackage{enumitem}
\numberwithin{equation}{section}
\usepackage[labelsep=endash]{caption}

\newcommand{\dd}{\mathrm{d}}

\DeclareMathOperator{\Aut}{Aut}

\DeclareMathOperator*{\Res}{Res}

\newcommand{\beq}{\begin{equation}}
\newcommand{\eeq}{\end{equation}}
\newcommand{\bea}{\begin{eqnarray}}
\newcommand{\eea}{\end{eqnarray}}

\theoremstyle{definition}
\newtheorem{thm}{Theorem}[section]

\newtheorem{defn}[thm]{Definition}

\newtheorem{rem}[thm]{Remark}
\newtheorem{lem}[thm]{Lemma}
\newtheorem{cor}[thm]{Corollary}
\newtheorem*{ex}{Example}

\newtheorem*{que*}{\textcolor{BrickRed}{Question}}

\def\bd{\begin{defn}}
\def\ed{\end{defn}}
\def\br{\begin{rem}}
\def\er{\end{rem}}
\def\bex{\begin{ex}}
\def\eex{\end{ex}}

\definecolor{webbrown}{rgb}{0.65, 0.16, 0.16}
\newcommand{\ben}{\begin{eqnarray*}}
\newcommand{\een}{\end{eqnarray*}}
\newcommand{\be}{\begin{equation}}
\newcommand{\ee}{\end{equation}}

\makeatletter
\renewenvironment{abstract}{%
    \if@twocolumn
      \section*{\abstractname}%
    \else \normalsize 
      \begin{center}%
        {\bfseries \normalsize\abstractname\vspace{\z@}}
      \end{center}%
      \quotation
    \fi}
    {\if@twocolumn\else\endquotation\fi}
\makeatother

\hypersetup{
colorlinks=true, linktocpage=true, pdfstartpage=1, pdfstartview=FitV,breaklinks=true, pdfpagemode=UseNone, pageanchor=true, pdfpagemode=UseOutlines,plainpages=false, bookmarksnumbered, bookmarksopen=true, bookmarksopenlevel=1,hypertexnames=true, pdfhighlight=/O,urlcolor=webbrown, linkcolor=RoyalBlue, citecolor=ForestGreen}

\title{\textsc{Topological recursion for fully simple maps from ciliated maps}}

\date{\vspace{-5ex}}
 
\author{Ga\"etan Borot\footnote{Institut f\"ur Mathematik \& Institut f\"ur Physik, Humboldt-Universit\"at zu Berlin, Rudower Chaussee 25, 12489 Berlin, Germany.}\,\,, S\'{e}verin Charbonnier\footnote{Max Planck Institut f\"ur Mathematik, Vivatsgasse 7, 53111 Bonn, Germany.}\,\,, Elba Garcia-Failde\footnote{Universit\'{e} de Paris, B\^{a}timent Sophie Germain, 8 Place Aur\'{e}lie Nemours, 75205 Paris Cedex 13, France.}}

\begin{document}

\maketitle

\vspace{2cm}

\begin{abstract}
Ordinary maps satisfy topological recursion for a certain spectral curve $(x,y)$. We solve a conjecture from \cite{BG-F18} that claims that fully simple maps, which are maps with non self-intersecting disjoint boundaries, satisfy topological recursion for the exchanged spectral curve $(y,x)$, making use of the topological recursion for ciliated maps \cite{BCEG-F}.
\end{abstract}

\thispagestyle{empty}
\bigskip

\tableofcontents

\newpage
\section*{Acknowledgements}
The authors are deeply grateful to Rapha\"el Belliard, Norman Do, Bertrand Eynard, Danilo Lewa\'nski and Ellena Moskovsky for the numerous discussions and computations on fully simple maps.

\smallskip

While finalising this work, the authors learned that B.~Bychkov, P.~Dunin-Barkowski, M.~Kazarian and S.~Shadrin had found an alternative proof for the fully simple conjecture. We are grateful for this communication and looking forward to exploring their work.

\smallskip

S.C.~was supported by the Max-Planck-Gesellschaft, and currently by the CAF of Paris. S.C.~wants to thank Liselotte Charbonnier and Lucie Neirac for material support. E.G.-F.~was supported by the public grant ``Jacques Hadamard'' as part of the Investissement d'avenir project, reference ANR-11-LABX-0056-LMH, LabEx LMH and currently receives funding from  the  European  Research Council  (ERC)  under  the  European  Union's Horizon  2020 research and  innovation  programme  (grant  agreement  No.~ERC-2016-STG 716083  ``CombiTop''). She is also grateful to the Institut des Hautes \'{E}tudes Scientifiques (IHES) for its hospitality.

\medskip

\newpage 

\section{Introduction: maps and fully simple maps}

A {\em map} $M$ of genus $g$ is the proper embedding of a finite connected graph into an oriented, topological, compact surface of genus $g$, so that the complement of the graph is a disjoint union of topological discs (called {\em faces}). We say that a face $f$ \emph{surrounds} a vertex $v$ (or an edge $e$) when $v$ (or $e$) belongs to the topological closure of $f$. Define an {\em oriented edge} to be an edge along with a choice of one of its two orientations. We say that an oriented edge is {\em adjacent} to a face if the face lies on its left and {\em incident} to a vertex if it points to this vertex. Maps may be endowed with the extra structure of an ordered tuple of distinct oriented edges, such that no two are adjacent to the same face. We refer to these oriented edges as {\em roots}, to their adjacent faces as {\em boundary faces}, and to all remaining faces as {\em internal faces}. The number of oriented edges adjacent to a face (resp. incident to a vertex) is called the {\em degree} of the face (resp. of the vertex). We denote by $\partial M$ the disjoint union of boundary faces and by $\partial_1 M, \ldots, \partial_n M$ the boundary faces ordered as their roots are. We say a map of genus $g$ with $n$ boundary faces has {\em topology $(g,n)$}. We say that the map is closed if it does not have boundary faces, \textit{i.e.} $n = 0$. Throughout the article, we will keep $g\geq 0$ and $n\geq 1$ unless stated otherwise.

\medskip

Two maps are equivalent if there exists an orientation-preserving homeomorphism between their underlying surfaces such that the vertices, oriented edges and faces of the first map are carried bijectively to the vertices, oriented edges and faces of the second, preserving the graph structure and the ordered tuple of roots. This also gives the notion of automorphism of a map, which is a permutation of the oriented edges arising from an orientation-pre\-ser\-ving ho\-meo\-morphism from the underlying surface to itself that preserves the graph structure and the tuple of roots\footnote{Note that for maps of topology $(g,n)$ with $n > 0$, we have ${\rm Aut}(M) = \{{\rm Id}\}$. Automorphisms therefore will not play a role in this article, although for conceptual reasons we prefer to keep $\#{\rm Aut}(M)$ in the formula for weights.}.

\medskip

$\mathbf{M}_{g,n}$ will denote the set of maps of genus $g$ with $n$ boundary faces and $\mathbf{M}_{g,n}(v)$ its subset of maps having $v$ vertices. The definition of a map allows different boundary faces to be adjacent along vertices and edges, as well as for a boundary face to be adjacent to itself along vertices and edges. Informally, we call a map fully simple if such behaviour does not arise --- a precise definition follows.

\begin{defn} \label{def:fullysimple}
An oriented edge in a map is a {\em boundary edge} if it is adjacent to a boundary face. A map is {\em fully simple} if at each vertex, at most one boundary edge is incident. \hfill $\star$
\end{defn}

In previous work, the term {\em simple} has been used to refer to maps in which for each $i \in \{1,\ldots,n\}$, at each vertex, at most one boundary edge adjacent to $\partial_iM$ is incident \cite{BG-F18}. We stress that, unlike in fully simple maps, boundary edges adjacent to distinct boundary faces can be incident to the same vertex. Throughout the article, we use the term {\em ordinary} to refer to the class of all maps, so as to emphasise the distinction from the class of fully simple maps. We will be interested primarily in the following enumerations of (equivalence classes of) ordinary and fully simple maps.

\begin{defn} \label{def:mapenumeration} 
Let $r\geq 2$ be an integer fixed throughout the article. For positive integers $k_1, k_2, \ldots, k_n$,  let
\[
\mathrm{Map}_{g;(k_1,\ldots,k_n)} \coloneqq \sum_{\substack{M\in\mathbf{M}_{g,n} \\ {\rm deg}(\partial_iM) = k_i}} \mathscr{W}(M)
\]
be the weighted enumeration of maps $M$ with $n$ boundary faces, such that the degree of the $i^{{\rm th}}$ boundary face is $k_i$ for $i \in \{1, 2, \ldots, n\}$. The weight of a map is given by 
\begin{equation}
\label{def:Wscr}\mathscr{W}(M) \coloneqq \frac{\alpha^{2-2g-\#\mathcal{V}(M)}}{\# \mathrm{Aut}(M)} \,  t_3^{f_3(M)} t_4^{f_4(M)} \cdots t_{r+1}^{f_{r+1}(M)}.
\end{equation}
Here, $f_k(M)$ is the number of internal faces of degree $k$, $\mathcal{V}(M)$ is the set of vertices, and $\mathrm{Aut}(M)$ the group of automorphisms. The analogous weighted enumeration restricted to the set of fully simple maps is denoted
\[
\mathrm{FSMap}_{g;(k_1,\ldots,k_n)} \coloneqq \sum_{\substack{M\in\mathbf{M}_{g,n} \\ {\rm fully}\,\,{\rm simple} \\ {\rm deg}(\partial_iM) = k_i}} \mathscr{W}(M).
\]
\hfill $\star$
\end{defn}

$\mathrm{Map}_{g;(k_1,\ldots,k_n)}$ and $\mathrm{FSMap}_{g;(k_1,\ldots,k_n)}$ are well-defined elements of $\alpha\mathbb{Z}\left[t_3,\ldots, t_{r+1}\right][[\alpha^{-1}]]$. For brevity, our notation makes implicit the dependence on the parameters $\alpha, t_3, \ldots, t_{r+1}$. From these formal power series, one can extract the number of maps with prescribed boundary face degrees, internal face degrees, and Euler characteristic. We adopt the convention that $\mathbf{M}_{0,1}(1)$ contains only the map consisting of a single vertex and no edges; it is the map of genus $0$ with $1$ boundary of length $0$, that is $\mathrm{Map}_{0;(0)}=\alpha$. Apart from this degenerate case, we always consider that boundaries have length $\geq 1$. For closed maps, there is no point in distinguishing ordinary and fully simple, and we denote ${\rm Map}_{g,\emptyset}$ the corresponding generating series.

\bd[\emph{Generating series of ordinary and fully simple maps}]\label{def:gen:fun:ord:fs}
Summing over all possible lengths, we define the generating series of maps of topology $(g,n)$ as follows:
\[
W_{g,n}(x_1,\ldots, x_n)\coloneqq \sum_{k_1,\ldots,k_n \geq 0} \frac{\mathrm{Map}_{g;(k_1,\ldots,k_n)}}{x_1^{1+k_1}\cdots x_n^{1+k_n}}.
\]
We have that $W_{g,n}(x_1,\ldots, x_n)\in \alpha\mathbb{Z}[x_1^{-1},\ldots,x_n^{-1},t_3,\ldots,t_{r+1}][[\alpha^{-1}]]$. We also introduce the generating series for fully simple maps of topology $(g,n)$:
\[
X_{g,n}(w_1,\ldots,w_n) \coloneqq \sum_{k_1,\ldots,k_n \geq 0} \mathrm{FSMap}_{g;(k_1,\ldots,k_n)} w_1^{k_1 - 1} \ldots w_n^{k_n - 1},
\]
and we have $X_{g,n}(w_1,\ldots, w_n)\in w_1^{-1}\cdots w_n^{-1}\mathbb{Z}[w_1,\ldots,w_n,t_3,\ldots,t_{r+1}][[\alpha^{-1}]]$. \hfill $\star$
\ed 

Functional relations determining the fully simple and the ordinary generating series for topology $(0,1)$ (discs) and $(0,2)$ (cylinders) have been established in \cite{BG-F18}, and the choice of $x_i^{-1}$ for ordinary maps but $w_i$ for fully simple maps makes the formulas more elegant. In particular, a straightforward adaptation of  \cite[Propositions 3.3. and 4.4]{BG-F18} to include the parameter $\alpha$ yields:
\begin{equation}
\label{01rel}
X_{0,1}(\alpha^{-1}W_{0,1}(x)) = \alpha x,\qquad W_{0,1}(\alpha^{-1}X_{0,1}(w)) = \alpha w
\end{equation}
and
\begin{equation}
\label{02rel}
\bigg(W_{0,2}(x_1,x_2) + \frac{1}{(x_1 - x_2)^2}\bigg)\dd x_1\dd x_2 = \bigg(X_{0,2}(w_1,w_2) + \frac{\alpha^2}{(w_1 - w_2)^2}\bigg)\dd w_1\dd w_2
\end{equation}
provided that $\alpha x_i = X_{0,1}(w_i)$ or equivalently $\alpha w_i = W_{0,1}(x_i)$.

\medskip

\noindent \textbf{Outline.} Our main result is Theorem~\ref{mainprop} establishing topological recursion for the fully simple generating series $X_{g,n}$. This proves the conjecture of \cite{BG-F18}, which was motivated by the transformations for discs and cylinders that can be justified by combinatorial methods and supported by numerical data in genus $1$. The present article builds on the enumerative study of combinatorial objects called \emph{multi-ciliated maps} carried out in \cite{BCEG-F} and reviewed in Section~\ref{S2}. We show in Section~\ref{S3} that multi-ciliated maps are dual to fully simple maps, and use in Section~\ref{S4} analytic techniques to establish the conjecture. This approach also gives a different proof of \eqref{01rel}-\eqref{02rel} (see Theorem~\ref{maincor}), \textit{i.e.} that the fully simple spectral curve (encoding $X_{0,1}$ and $X_{0,2}$) and the ordinary spectral curve (encoding $W_{0,1}$ and $W_{0,2}$) are related by the symplectic transformation "exchange of $x$ and $y$``. We discuss a few consequences (some of them anticipated in \cite{BG-F18} conditionally to the former conjecture) in Section~\ref{S5}.

\section{Multi-ciliated maps}
\label{S2}

In \cite{BCEG-F}, we introduced various types of generalised Kontsevich graphs. Among them, the ciliated and multi-ciliated maps are the types that we shall relate to fully simple maps. We remind the reader about their definitions; we also define their weights and their generating series; last, we present preliminary results coming from \cite{BCEG-F}.

\medskip

Ciliated and multi-ciliated maps are maps with two types of vertices (black and white), which satisfy specific constraints.

\bd[\emph{Constraints on the vertices}]\label{def:const:vertices}
\emph{Black} and \emph{white} vertices have the following properties:
\begin{itemize}
	\item a black vertex $v_{\bullet}$ must have a degree $ \deg(v_{\bullet}) \in \{3,\ldots,r + 1\}$;
	\item a white vertex may have arbitrary positive degree;
	\item \emph{star constraint}: for any given white vertex $v_{\circ}$, the faces surrounding $v_{\circ}$ are pairwise distinct;
	\item \emph{uniqueness constraint}: each face surrounds at most one white vertex.
\hfill $\star$
\end{itemize}
\ed 
Note that a univalent white vertex automatically satisfies the star constraint.

\medskip

\bd[\emph{Ciliated maps}]\label{def:cil:konts}
$M\in\mathbf{C}_{g,n}$ if:
\begin{itemize}
	\item $M$ is a connected map of genus $g$;
	\item $M$ has exactly $n$ white vertices, labelled from $1$ to $n$. They are univalent and the face surrounding the $i^{\textup{th}}$ white vertex is called the $i^{\textup{th}}$ \emph{marked face}. The faces that do not surround a white vertex are called \emph{unmarked faces}.
	\end{itemize}
Every such map is called \emph{ciliated map of type $(g,n)$}. Note that in this case, $\Aut(M)=\{\rm{Id}\}$. 
\hfill $\star$
\ed
The term \emph{ciliated} comes from the fact that each of the $n$ white vertices has degree $1$, \textit{i.e.}~there is only one oriented edge incident to it. This edge attached to a white vertex can be viewed as a cilium in the corresponding marked face.

\medskip

\bd[\emph{Multi-ciliated maps}]\label{def:multicil:konts}
Let $\underline{k} = (k_1,\ldots,k_n)$ be an $n$-tuple of positive integers.  $M\in\mathbf{S}_{g,\underline{k}}$ if:
\begin{itemize}
	\item $M$ is a connected map of genus $g$;
	\item $M$ has $n$ white vertices, labelled from $1$ to $n$. Each of them equipped with a choice of incident edge (the $i^{{\rm th}}$ root). The $i^{\textup{th}}$ white vertex has degree $k_i$ and is equipped with a choice of incident edge (the \emph{$i^{{\rm th}}$ root}). The faces surrounding white vertices are called \emph{marked faces}, while the other are called \emph{unmarked faces}.  The set of unmarked faces is denoted $\mathcal{F}^{\circ}(M)$. 
\end{itemize}
Every such map is called \emph{multi-ciliated map of type $(g,n)$}. Again, $\Aut(M)=\{\rm{Id}\}$. \hfill $\star$
\ed
Comparing Definitions~\ref{def:cil:konts} and \ref{def:multicil:konts}, we have $\mathbf{C}_{g,n} =\mathbf{S}_{g;(1,\ldots,1)}$. 
\bd[\emph{Degree of a map}]
\label{def:degree}We define the \emph{degree} of a map $M$ as:
\[
\deg M \coloneqq \#\mathcal{E}(M)-  \#\mathcal{V}(M) = 2g(M) - 2 + \#\mathcal{F}(M),
\]
where $g(M)$ is the genus of $M$ and $\mathcal{V}(M)$, $\mathcal{E}(M)$ and $\mathcal{F}(M)$ respectively the sets of vertices, edges and faces of the underlying graph. We denote by
$$
\mathbf{C}_{g,n}^{\delta} \subseteq \mathbf{C}_{g,n}  \qquad {\rm and}\qquad \mathbf{S}_{g,\underline{k}}^{\delta} \subseteq  \mathbf{S}_{g,\underline{k}}
$$ 
the corresponding subsets of maps of fixed degree $\delta$.\hfill $\star$
\ed 
An easy Euler counting (see \textit{e.g.}~\cite{BCEG-F}) shows that for a given topology $(g,n)$ and degree $\delta=(2g-2 + \#\mathcal{F})$, the sets $\mathbf{C}_{g,n}^{\delta}$ and $\mathbf{S}_{g,\underline{k}}^{\delta}$ are finite. We now turn to the local weights of a (multi)-ciliated map. 

\bd[\emph{Local weights}]\label{def:potential}
The \emph{potential} of the model is a polynomial of degree $r+1$:
\begin{equation*}
V(u) \coloneqq \frac{u^2}{2}-\sum_{j=3}^{r+1} \frac{t_j}{j}\,u^j.
\end{equation*}
We then introduce the \emph{propagator}:
\begin{equation*}
\mathscr{P}(u_1,u_2)\coloneqq\frac{u_1-u_2}{V'(u_1)-V'(u_2)},
\end{equation*}
and for $d \in \{3,\ldots,r + 1\}$:
\begin{equation*}
\mathscr{V}_d(u_{1},\ldots,u_{d})
\coloneqq \Res_{u = \infty} \frac{V'(u)\dd u}{\prod_{j = 1}^d (u - u_j)} = \sum_{i=1}^d \frac{-\,V'(u_i)}{\prod_{j\neq i} (u_i-u_j)}.
\end{equation*}
It is manifest from the first formula that this weight is a symmetric polynomial in $u_1,\ldots,u_d$. In particular it is well-defined when some of the $u_i$ coincide.
\hfill $\star$
\ed

\medskip

We fix a finite set of complex numbers $\Lambda = \{\lambda_1,\ldots,\lambda_N\}$, considered as parameters. For a given $n$-tuple $\underline{k}$, we denote $Z_i = [z_{i,1},\ldots,z_{i,k_i}]$ a $k_i$-tuple of complex variables, and $\underline{Z} = (Z_1,\ldots,Z_n)$ --- square brackets are used for better parsing. It is now possible to associate a weight to a (multi)-ciliated map, by summing over decorations of unmarked faces and multiplying the local weights:

\bd
\label{def:wcil}The \emph{weight} of a (multi-)ciliated map $M$ is given by:
\begin{equation*}
\mathscr{W}_{\rm cil}(M) =\sum_{U\,:\,\mathcal{F}^{\circ} \rightarrow \Lambda} 
\prod_{\substack{e\in \mathcal{E}(M)\\ e=(f_1,f_2)}} \mathscr{P}(u_{f_1},u_{f_2})
\prod_{\substack{v\in\mathcal{V}(M)\\ \mathrm{black}}}\,\mathscr{V}_{\deg (v)}(\{u_f\}_{f\mapsto v}),
\end{equation*}
Here, $u_f$ is the decoration of a face: for an unmarked face it is $u_f = U(f) \in \Lambda$ ; for marked faces, it is $z_{i,1}$  for the face adjacent to the $i^{{\rm th}}$ root, and starting from this one, $z_{i,2},\ldots,z_{i,k}$ for the faces encountered when travelling anticlockwise around the $i^{{\rm th}}$ vertex.  The notation $e = (f_1,\,f_2)$ means that $e$ is the edge surrounded by the faces $f_1$ and $f_2$, and $f\mapsto v$ means that $f$ surrounds $v$.
\hfill $\star$
\ed

Note that white vertices have weight $1$ in this formula, and that if $\Lambda = \emptyset$, the (multi-)ciliated maps $M$ with $\mathcal{F}^{\circ}(M)\neq \emptyset$ have vanishing weight: $\mathscr{W}_{\rm cil}(M) =0$.

\bd[\emph{Generating series of (multi-)ciliated maps}]\label{def:generating:functions}
\begin{eqnarray*} 
C_{g,n}(Z) &=&\sum\limits_{M\in\mathbf{C}_{g,n}} \alpha^{-\deg M}\,\mathscr{W}_{\rm cil}(M) \\
 &=& \sum\limits_{\delta\geq 2g - 2  + n}\alpha^{-\delta} \sum\limits_{M\in\mathbf{C}_{g,n}^{\delta}} \mathscr{W}_{\rm cil}(M), \\
S_{g;\underline{k}}(\underline{Z}) &=&\sum\limits_{M\in\mathbf{S}_{g;\underline{k}}} \alpha^{-\deg M}\,\mathscr{W}_{\rm cil}(M) \\ 
 &=& \sum\limits_{\delta\geq 2g - 2 + n}\alpha^{-\delta} \sum\limits_{M\in\mathbf{S}_{g;\underline{k}}^{\delta}} \mathscr{W}_{\rm cil}(M).
\end{eqnarray*}  
\hfill $\star$
\ed
Those generating series are well-defined formal series in $\alpha^{-1}$, except $C_{0,1}$ which is in addition contains the term $\alpha$. The dependence on $\alpha^{-1}$ and $\lambda$s and $t$s has been omitted from the notation. We now recall a key recursion on the degree of white vertices for multi-ciliated generating series.

\begin{lem} \cite{BCEG-F}
\label{lem:multicil:to:cil} If $k_1 \geq 2$, set $\underline{k}'=(k_1-1,k_2,\ldots,k_n)$.  We have:
\begin{equation*}
\begin{split}
S_{g;\underline{k}}(Z_1,\ldots,Z_n)&=\frac{1}{\alpha}\frac{S_{g;\underline{k}'}\left([z_{1,1},z_{1,3},\ldots,z_{1,k_1}],Z_2,\ldots,Z_n\right)-S_{g;\underline{k}'}\left([z_{1,2},\ldots,z_{1,k_1}],Z_2,\ldots,Z_n\right)}{V'(z_{1,1})-V'(z_{1,2})} \\
& \quad + \delta_{g,0}\delta_{n,1}\delta_{k_1,2}\mathscr{P}(z_{1,1},z_{1,2}).
\end{split}
\end{equation*}
 \hfill $\star$
\end{lem}
Applying this formula recursively, we can express the generating series of multi-ciliated maps in terms of the generating series of ciliated maps:
\begin{lem} \cite{BCEG-F} \label{lem:multicil:to:cil:1} Recall that $Z_i=\left[z_{i,1},\ldots,z_{i,k_i}\right]$ for $i \in \{1,\ldots,n\}$. We have:
\begin{equation}\label{eq:multicil:to:cil:3}
S_{g;\underline{k}}(Z_1,\ldots,Z_n) = \frac{1}{\alpha^{k_1+\ldots+k_n-n}}  \sum\limits_{j_1=1}^{k_1}\cdots\sum\limits_{j_n=1}^{k_n} \frac{C_{g,n}(z_{1,j_1},\ldots,z_{n,j_n}) + \delta_{g,0}\delta_{n,1} \alpha z_{1,j_1}}{\prod\limits_{i=1}^{n}\prod\limits_{\substack{l_i=1\\ l_i\neq j_i}}^{k_i} \big(V'(z_{i,j_i})-V'(z_{i,l_i})\big)}. 
\end{equation}
\hfill $\star$
\end{lem}

\section{Relating fully simple maps and multi-ciliated maps} \label{sec:preliminaries}
\label{S3}

The aim of this section is to show that multi-ciliated maps are dual to fully simple maps. It is convenient to do so via the permutation model for maps presented \textit{e.g.} in \cite{lan-zvo04}.

\subsection{Permutation model for fully simple maps} \label{subsec:maps}

An unrooted map can be encoded into a triple $(\sigma_0, \sigma_1, \sigma_2)$ of permutations acting on the set $\vec{\mathcal{E}}$ of oriented edges, in which
\begin{itemize}
\item $\sigma_0$ rotates each oriented edge anticlockwise around the vertex it is incident to;
\item $\sigma_1$ is the fixed-point-free involution swapping the two oriented edges with same underlying edge; 
\item $\sigma_2$ rotates each oriented edge anticlockwise around the face it is adjacent to (\textit{i.e.} located to its left).
\end{itemize}
The vertices, edges and faces of the map correspond respectively to the cycles of $\sigma_0$, $\sigma_1$ and $\sigma_2$.

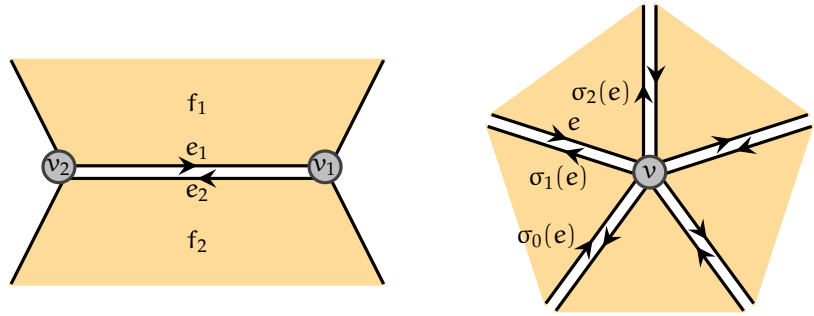
\begin{figure}[ht!]
\centering
\begin{tikzpicture}[scale=0.7]
\begin{scope}[very thick,decoration={markings, mark=at position 0.5 with {\arrow[scale=1.4]{stealth}}}] 
	\filldraw[orange!40!white] (-1,2.12) -- (0,0.12) -- (5,0.12) -- (6,2.12) -- cycle;
	\filldraw[orange!40!white] (-1,-2.12) -- (0,-0.12) -- (5,0-.12) -- (6,-2.12) -- cycle;
	\draw[postaction={decorate}] (0,0.12) -- (5,0.12);
	\draw[postaction={decorate}] (5,-0.12) -- (0,-0.12);
	\draw (0,0.12) -- (-1,2.12);
	\draw (5,0.12) -- (6,2.12);
	\draw (0,-0.12) -- (-1,-2.12);
	\draw (5,-0.12) -- (6,-2.12);
	\node at (2.5,0.4) {$e_1$};
	\node at (2.5,-0.4) {$e_2$};
	\node at (2.5,1.3) {$f_1$};
	\node at (2.5,-1.3) {$f_2$};	
	\begin{scope}[shift={(-0.1,0.1)}]
		\filldraw[fill=lightgray,draw=darkgray] (0,0) circle (0.3);
		\node at (0,0) {$v_2$};
	\end{scope}
	\begin{scope}[shift={(-0.1,0.1)}]
		\filldraw[fill=lightgray,draw=darkgray] (5,0) circle (0.3);
		\node at (5,0) {$v_1$};
	\end{scope}
\end{scope}
\begin{scope}[shift={(11,0)},very thick,decoration={markings, mark=at position 0.5 with {\arrow[scale=1.4]{stealth}}}] 
	\begin{scope}[rotate=18]
		\filldraw[orange!40!white,shift={(36:0.2)}] (0:3) -- (0,0) -- (72:3) -- cycle;
		\draw[postaction={decorate},shift={(36:0.2)}] (0,0) -- (0:3);
		\draw[postaction={decorate},shift={(36:0.2)}] (72:3) -- (0,0);
	\end{scope}
	\begin{scope}[rotate=90]
		\filldraw[orange!40!white,shift={(36:0.2)}] (0:3) -- (0,0) -- (72:3) -- cycle;
		\draw[postaction={decorate},shift={(36:0.2)}] (0,0) -- (0:3);
		\draw[postaction={decorate},shift={(36:0.2)}] (72:3) -- (0,0);
	\end{scope}
	\begin{scope}[rotate=162]
		\filldraw[orange!40!white,shift={(36:0.2)}] (0:3) -- (0,0) -- (72:3) -- cycle;
		\draw[postaction={decorate},shift={(36:0.2)}] (0,0) -- (0:3);
		\draw[postaction={decorate},shift={(36:0.2)}] (72:3) -- (0,0);
	\end{scope}
	\begin{scope}[rotate=234]
		\filldraw[orange!40!white,shift={(36:0.2)}] (0:3) -- (0,0) -- (72:3) -- cycle;
		\draw[postaction={decorate},shift={(36:0.2)}] (0,0) -- (0:3);
		\draw[postaction={decorate},shift={(36:0.2)}] (72:3) -- (0,0);
	\end{scope}
	\begin{scope}[rotate=306]
		\filldraw[orange!40!white,shift={(36:0.2)}] (0:3) -- (0,0) -- (72:3) -- cycle;
		\draw[postaction={decorate},shift={(36:0.2)}] (0,0) -- (0:3);
		\draw[postaction={decorate},shift={(36:0.2)}] (72:3) -- (0,0);
	\end{scope}
	\filldraw[fill=lightgray,draw=darkgray] (0,0) circle (0.3);
	\node at (0,0) {$v$};
	\node[above,shift={(0,0.1)}] at (162:1.5) {$e$};
	\node[left,shift={(-0.1,0)}] at (90:1.5) {$\sigma_2(e)$};
	\node[below,shift={(-0.2,-0.1)}] at (162:1.5) {$\sigma_1(e)$};
	\node[left,shift={(-0.2,0)}] at (234:1.5) {$\sigma_0(e)$};
\end{scope}
\end{tikzpicture}
\caption{The left panel depicts the local structure of an edge in a map. The oriented edges $e_1$ and $e_2$ are indicated by the arrows. With our conventions, $e_i$ is {\em adjacent} to face $f_i$ and {\em incident} to vertex $v_i$ for $i = 1, 2$. The right panel depicts the local structure of a vertex in a map, including the action of the permutations $\sigma_0, \sigma_1, \sigma_2$ on an oriented edge $e$.}
\label{fig:vertex}
\end{figure}

It follows that $\sigma_0 \circ \sigma_1 \circ \sigma_2 = \mathrm{Id}$. This can easily be adapted to describe rooted maps.

\begin{lem}
A rooted map can be encoded by a triple $(\sigma_0, \sigma_1, \sigma_2)$ of permutations in $\mathfrak{S}(\vec{\mathcal{E}})$ and a tuple $\mathcal{R} \in \vec{\mathcal{E}}^n$ such that
\begin{itemize}
\item $\sigma_1$ is a fixed-point-free involution;
\item $\sigma_0 \circ \sigma_1 \circ \sigma_2 = \mathrm{Id}$;
\item no two elements of $\mathcal{R}$ belong to the same cycle of $\sigma_2$.
\end{itemize}
The data $(\sigma_0, \sigma_1, \sigma_2; \mathcal{R})$ and $(\sigma'_0, \sigma'_1, \sigma'_2; \mathcal{R}')$ define equivalent maps if and only if there exists a bijection $\phi: \vec{\mathcal{E}} \to \vec{\mathcal{E}}'$ that sends $\mathcal{R}$ to $\mathcal{R}'$ and satisfies $\sigma'_i = \phi  \circ \sigma_i \circ \phi^{-1}$ for $i \in \{0, 1, 2\}$. \hfill $\star$
\end{lem}

\subsubsection{Characterisation of simplicity in the permutation model}
This permutation model allows the following characterisation of simple maps. Suppose that a map is given by the data $(\sigma_0, \sigma_1, \sigma_2 ; \mathcal{R})$ with  $\mathcal{R}=(e_1,\ldots,e_n)$. For $i\in\{1,\ldots,n\}$, define the set $\mathcal{B}_i \subseteq \vec{\mathcal{E}}$ to be the $\sigma_2$-orbits of $e_i$. We observe that $\mathcal{B}_i$ naturally corresponds to the set of boundary edges around the $i^{\textup{th}}$ marked face, and this face is simple if and only if the elements of $\mathcal{B}_i$ belong to pairwise distinct $\sigma_0$-orbits.

\medskip

Let us describe the characterisation of simple maps in a slightly different way, using a notation that will later be useful. For $i\in\{1,\ldots,n\}$, denote by $\sigma_0^{\partial_i} \in \mathfrak{S}(\mathcal{B}_i)$ the permutation obtained by expressing $\sigma_0 \in \mathfrak{S}(\vec{\mathcal{E}})$ as a union of disjoint cycles and deleting those elements that do not lie in $\mathcal{B}_i$. If $e \in \mathcal{B}_i$ is an oriented edge incident to the vertex $v$, then $\sigma_0^{\partial_i}(e)$ is the next oriented edge in $\mathcal{B}_i$ incident to $v$ that is encountered when turning anticlockwise around $v$. Then the $i^{\textup{th}}$ boundary face is simple if and only if the permutation $\sigma_0^{\partial_i}$ is the identity permutation.

\subsubsection{Characterisation of full simplicity in the permutation model}

For fully simple maps, define the set $\mathcal{B} \subseteq \vec{\mathcal{E}}$ to be the union of the $\sigma_2$-orbits of the elements of $\mathcal{R}$. We observe that $\mathcal{B}$ naturally corresponds to the set of boundary edges. Then, the map is fully simple if and only if the elements of $\mathcal{B}$ belong to pairwise distinct $\sigma_0$-orbits. Equivalently, denote by $\sigma_0^\partial \in \mathfrak{S}(\mathcal{B})$ the permutation obtained by expressing $\sigma_0 \in \mathfrak{S}(\vec{\mathcal{E}})$ as a union of disjoint cycles and deleting those elements that do not lie in $\mathcal{B}$. If $e \in \mathcal{B}$ is an oriented edge incident to the vertex $v$, then $\sigma_0^\partial(e)$ is the next oriented edge in $\mathcal{B}$ incident to $v$ that is encountered when turning anticlockwise around $v$. A map is then fully simple if and only if the permutation $\sigma_0^\partial$ is the identity permutation.

\subsection{Characterisation of multi-ciliated maps in the permutation model}\label{staruni}
Next, we show that the star constraint concept on white vertices of multi-ciliated maps is the dual of simplicity, and that adding furthemore the uniqueness constraint, we get the dual concept of full simplicity.

\begin{lem}
A multi-ciliated map with $n$ white vertices is encoded into a triple $(\sigma_0', \sigma_1',\sigma_2')$ of permutations in $\mathfrak{S}(\vec{\mathcal{E}}')$ and a tuple $\mathcal{R}' \in (\vec{\mathcal{E}}')^n$ such that
\begin{itemize}
\item $\sigma_1'$ is a fixed-point-free involution;
\item $\sigma_0' \circ \sigma_1' \circ \sigma_2' = \mathrm{Id}$; and
\item no two elements of $\mathcal{R}'$ lie in the same cycle of $\sigma_0'$.
\end{itemize}
The data $(\sigma_0',\sigma_1', \sigma_2'; \mathcal{R}')$ and $(\sigma_0'', \sigma_1'',\sigma_2''; \mathcal{R}'')$ define equivalent maps if and only if there exists a bijection $\phi\,:\,\vec{\mathcal{E}}' \rightarrow \vec{\mathcal{E}}''$ that sends $\mathcal{R}'$ to $\mathcal{R}''$ and satisfies $\sigma''_i = \phi \circ \sigma_i' \circ \phi^{-1}$ for $i \in \{0, 1, 2\}$. \hfill $\star$
\end{lem}
\begin{proof}

We take $(\sigma_0',\sigma_1',\sigma_2')$ the triple of permutations encoding the underlying map. For each $i \in \{1,\ldots,n\}$, we take $e_i$ to be the $i^{{\rm th}}$ root. We then set $\mathcal{R}' = (e_1,\ldots,e_n)$. By construction all the  conditions announced in the lemma are satisfied.
\end{proof}

In the rest of the subsection, multi-ciliated maps $M'$ are replaced by a corresponding permutation model $(\sigma'_0, \sigma'_1,\sigma_2';\mathcal{R}')$ with $\mathcal{R}' =(e_1,\ldots,e_n)$.

\subsubsection{Star constraint}

The permutation model allows the following characterisation of the star-constraint. For $i\in\{1,\ldots,n\}$, define the set $\mathcal{B}_i' \subseteq \vec{\mathcal{E}}'$ to be the $\sigma_0'$-orbit of the $e_i$. We observe that it corresponds naturally to the set of oriented edges incident to the $i^{\textup{th}}$ white vertex. Denote by $(\sigma_2')^{\partial_i} \in \mathfrak{S}(\mathcal{B}_i')$ the permutation obtained by expressing $\sigma_2' \in \mathfrak{S}(\vec{\mathcal{E}}')$ as a union of disjoint cycles and deleting those elements that do not lie in $\mathcal{B}'_i$. If $e \in \mathcal{B}_i'$ is an oriented edge adjacent to a marked face $f$, then $(\sigma_2')^{\partial_i}(e)$ is the next oriented edge in $\mathcal{B}_i'$ met when turning anticlockwise around $f$. The $i^{\textup{th}}$ white vertex satisfies the star constraint if and only if the permutation $(\sigma_2')^{\partial_i}$ is the identity permutation.

\begin{lem}[Star = dual of simplicity]
The $i^{\textup{th}}$ white vertex of a multi-ciliated map satisfies the star constraint if and only if the $i^{\textup{th}}$ boundary face of the dual map is simple. \hfill $\star$
\end{lem}
\begin{proof}
Let $M'=(\sigma_2',\sigma'_1,\sigma_0'; \mathcal{R}')$ be a multi-ciliated map of type $(g,n)$. Define the map $M=(\sigma_0,\sigma_1,\sigma_2;\mathcal{R})$ $=(\sigma_2',\sigma'_1,\sigma_0'; \mathcal{R}')$ as the dual of $M'$. The white vertices of $M'$ correspond to the boundary faces of $M$; the black ones correspond to the internal faces of $M$. Then:
\begin{itemize}
	\item $M$ is of genus $g$ since $M'$ is of genus $g$;
	\item $M$ has $n$ labelled boundary faces since $M'$ has $n$ labelled white vertices;
	\item the boundary faces of $M$ are rooted: the root of the $i^{\textup{th}}$ boundary face is the dual oriented edge to the $i^{{\rm th}}$ root edge of $M'$.
\end{itemize}

The $i^{\textup{th}}$ white vertex of $M'$ satisfies the star constraint if and only if the permutation $(\sigma_2')^{\partial_i}$ is the identity permutation, \textit{i.e.}~the permutation $\sigma_0^{\partial_i}$ is the identity permutation. This is the property defining the simplicity of the $i^{\textup{th}}$ boundary face.
\end{proof}

\subsubsection{Uniqueness constraint}

In the permutation model, having the star and the uniqueness constraints simultaneously is characterised as follows. Define $\mathcal{B}' \subseteq \vec{\mathcal{E}}'$ to be the union of the $\sigma_0'$-orbits of $e_i$, for $i \in \{1,\ldots,n\}$.  Denote by $(\sigma_2')^{\partial} \in \mathfrak{S}(\mathcal{B}')$ the permutation obtained by expressing $\sigma_2' \in \mathfrak{S}(\vec{\mathcal{E}}')$ as a union of disjoint cycles and deleting those elements that do not lie in $\mathcal{B}'$. If $e \in \mathcal{B}'$ is an oriented edge adjacent to a marked face $f$, then $(\sigma_2')^{\partial}(e)$ is the next oriented edge in $\mathcal{B}'$ that is met when turning anticlockwise around $f$. Then the vertices of a multi-ciliated map satisfy the star and uniqueness constraints if and only if the permutation $(\sigma_2')^{\partial}$ is the identity permutation.

\begin{lem}[Uniqueness and star = dual of full simplicity]\label{lem:perm:fs}
The white vertices of a multi-ciliated map satisfy the uniqueness and star constraints if and only if the dual map is fully simple. \hfill $\star$
\end{lem}
\begin{proof}
As in the previous proof, the property that $\sigma_0^{\partial} = (\sigma_2')^{\partial}$ is the identity permutation matches the definition of full-simplicity.
\end{proof}

\section{Topological recursion for fully simple maps}

\label{S4}

We use multi-ciliated maps in order to prove that fully simple maps satisfy topological recursion. We first recall from \cite{BCEG-F} the topological recursion formula for ciliated maps; we then show how to use this result to prove that fully simple maps satisfy topological recursion; eventually, we discuss in greater detail the disc and the cylinder case.

\subsection{Topological recursion for ciliated maps}

If $f$ is a formal power series in $\alpha^{-1}$, we denote by $[f]_d$ the coefficient of $\alpha^{-d}$. We define the fol\-low\-ing spectral curve, which is a specialisation of the spectral curve for ciliated maps obtained in \cite{BCEG-F}.
\begin{lem} \cite{BCEG-F} \label{def:spec:curve}
There exists a unique polynomial $Q$ of degree $r$ with coefficients in $\mathbb{C}[[\alpha^{-1}]]$, as well as $a_j\in\mathbb{C}[[\alpha^{-1}]]$ indexed by $j\in\{1,\ldots,N\}$, satisfying:
\begin{equation*}
\begin{array}{lll}
Q(\zeta) \mathop{=}\limits_{\zeta \rightarrow \infty} V'\bigg(\zeta + \frac{1}{\alpha}\sum\limits_{j=1}^{N}\frac{1}{Q'(a_j)(\zeta-a_j)} \bigg)+\mathcal{O}(\zeta^{-1}),&\qquad& [Q(\zeta)]_0=V'(\zeta),\vspace{0.15cm} \\
Q(a_j) = V'(\lambda_j),&\qquad& [a_j]_0=\lambda_j.
\end{array}
\end{equation*}
\hfill $\star$ 
\end{lem}

\bd
We introduce $\zeta\in\mathbb{C}[[\alpha^{-1}]]$, the unique formal power series whose coefficients are rational functions of $z$ determined by:
\begin{equation}
\label{Qzeta} Q(\zeta) = V'(z),\qquad [\zeta]_0 = z. 
\end{equation}
The \emph{spectral curve} for the weighted enumeration of ciliated maps is $\mathcal{S}=(\mathbb{P}^1,x,y,\Gamma_{0,2})$, where the meromorphic maps $x,y\,\colon\, \mathbb{P}^1\to \mathbb{P}^1$ and the bidifferential $\Gamma_{0,2}$ are defined by:
\begin{equation*}
\begin{cases}
x(\zeta)= Q(\zeta), \\
y(\zeta)= \alpha\,\zeta +\sum\limits_{j=1}^N \frac{1}{Q'(a_j)(\zeta-a_j)}, \\
\Gamma_{0,1}(\zeta)=y(\zeta)\dd x(\zeta), \vspace{0.1cm} \\
\Gamma_{0,2}(\zeta_1,\zeta_2) = \frac{\dd\zeta_1 \, \dd\zeta_2}{(\zeta_1-\zeta_2)^2}.
\end{cases}
\end{equation*}\hfill
$\star$
\ed
Here, $\zeta$ is considered as a uniformising coordinate on the underlying $\mathbb{P}^1$, and $\Gamma_{0,2}$ is the unique fundamental bidifferential of the second kind on $\mathbb{P}^1$ (called \emph{standard bidifferential} for short). To be precise, $\mathcal{S}$ is a family of spectral curves parameterised by the formal parameter $\alpha^{-1}$ and the complex parameters $\lambda_1,\ldots,\lambda_N,t_3,\ldots,t_r$. 

\medskip

For $g \geq 0$ and $n \geq 1$ such that $2g-2+n \geq 0$, we define the $n$-differential
\begin{eqnarray}\label{gammaTR}
\Gamma_{g,n}(\zeta_1,\ldots,\zeta_n) &=& \bigg(C_{g,n}(z_1,\ldots,z_n) + \frac{\delta_{g,0}\delta_{n,2}}{\big(x(\zeta_1) - x(\zeta_2)\big)^2}\bigg)\dd x(\zeta_1)\cdots \dd x(\zeta_n)\cr
&& +\,\,\delta_{g,0}\delta_{n,1}\bigg(\alpha\,z_1+\sum\limits_{j=1}^{N}\frac{1}{V'(z_1)-V'(\lambda_j)}\bigg)\dd x(\zeta_1),
\end{eqnarray}
where $\zeta_i$ and $z_i$ are related as in Definition~\ref{Qzeta}. It was proved in \cite{BCEG-F} that ciliated maps satisfy topological recursion for this curve.

\begin{thm} \cite{BCEG-F} \label{thm:tr:ciliated}
If $\lambda_1,\ldots,\lambda_N$ and $t_3,\ldots,t_{r+1}$ are chosen so that the polynomial $Q'(\zeta)$ has simple roots (this is a generic condition), then the differentials $\Gamma_{g,n}$ can be analytically continued to meromorphic $n$-forms on the spectral curve, still denoted $\Gamma_{g,n}$. Then, $\Gamma_{0,2}$ is then the standard bidifferential on the spectral curve, and $\Gamma_{g,n}$ for $2g - 2 + n > 0$ satisfy the topological recursion for the spectral curve of Definition \ref{def:spec:curve}. The generating series $C_{g,n}$ are retrieved by expansion when $z_i \rightarrow \infty$. \hfill $\star$
\end{thm}

The roots of the polynomial $Q'(\zeta)$ are the \emph{ramification points} of the spectral curve. Given $\zeta_0$, we define $\{\zeta_0^{(0)},\zeta_0^{(1)},\ldots,\zeta_0^{(r-1)}\}$ as the set of roots of $Q(\zeta) -Q(\zeta_0)$, where $\zeta_0^{(0)}=\zeta_0$.  For generic parameters, the branched cover $x$ has $r-1$ simple ramification points $b_1,\ldots,b_{r-1}$, \textit{i.e.}
$$
Q'(b_k)=0,\qquad Q''(b_k) \neq 0,
$$
and hence the theorem applies. For $\zeta$ near $\rho_k$, we can always choose the labellings of points in $x^{-1}(x(\zeta))$ so that $\rho_k = \rho_k^{(k)}$; since the ramification points are simple, we have $\rho_k^{(l)} \neq \rho_k$ for $l \neq 0,k$.  To each ramification point we associate the recursion kernel:
\[K_{\rho_k} (\zeta_1,\zeta) = \frac{1}{2}\,\frac{\int_{\zeta^{(k)}}^{\zeta} \Gamma_{0,2}(\zeta_1,\cdot)}{\Gamma_{0,1}(\zeta)- \Gamma_{0,1}(\zeta^{(k)})}\;, \] 
which is defined locally for $\zeta$ near $\rho_k$ and globally for $\zeta_1 \in \mathbb{P}^1$. The topological recursion formula allows the computation of $\Gamma_{g,n}$ by induction on $2g-2+n > 0$:
\begin{equation}
\label{gammaTR1}
\Gamma_{g,n}(\zeta_1,\ldots,\zeta_n) = \sum\limits_{k=1}^{r-1}\underset{\zeta=\rho_k}{\Res}K_{\rho_k}(\zeta_1,\zeta)\Bigg(\Gamma_{g-1,n+1}(\zeta,\zeta^{(k)},I)+\!\!\!\!\sum\limits_{\substack{h+h'=g\\ J\sqcup J' =I}}^{'} \!\!\!\!\Gamma_{h,1+\#J}(\zeta,J)\,\Gamma_{h',1+\#J'}(\zeta^{(k)},J')\Bigg),
\end{equation}
where $I = \{\zeta_2,\ldots,\zeta_n\}$ and $\sum'$ means that terms involving $\Gamma_{0,1}$ should be excluded from the sum.

\subsection{Relating fully simple and ciliated generating series}

The following key observation relates the enumeration of fully simple maps to the one of multi-ciliated maps. If $k \geq 0$, let $[0^{k}]$ be the $k$-tuple whose elements are all zero. \textit{From now on we set $\Lambda = \{0\}$, \textit{i.e.}~$N = 1$ and $\lambda_1 = 0$.}

\begin{lem}
If we choose $\Lambda = \{0\}$, we have for any $g \geq 0$, $n \geq 1$ and $k_1,\ldots,k_n \geq 0$:
\begin{equation}\label{eq:multi:fs}
\mathrm{FSMap}_{g;(k_1,\ldots,k_n)} = \left. S_{g;(k_1,\ldots,k_n)}\big([0^{k_1}],\ldots,[0^{k_n}]\big)\right|_{\lambda_1 = 0} +\delta_{g,0}\delta_{n,1}\delta_{k_1,0} \alpha. 
\end{equation}
\hfill $\star$
\end{lem}
\br
The additional term for the disc case comes from the degenerate fully simple map in $\mathbf{M}_{0,1}(1)$, which has no equivalent among multi-ciliated maps.
\hfill $\star$
\er
\begin{proof} From Lemma \ref{lem:perm:fs}, multi-ciliated maps are dual to fully simple maps --- we rigorously characterised this correspondence in Section~\ref{staruni}. The perimeter of the $i^{\textup{th}}$ boundary face of a fully simple map $M\in\mathbf{M}_{g,n}$ corresponds to the degree of the $i^{\textup{th}}$ white vertex of $M' \in\mathbf{S}_{g;(k_1,\ldots,k_n)}$, \textit{i.e.}~$\deg(\partial_i M)=k_i$. We shall now compare the weights in their enumeration. Recall the Definition~\ref{def:potential} for the potential $V$ and $\mathscr{V}_d$, especially its expression via a residue. We observe that 
\begin{equation*}
\begin{split} 
\left. \mathscr{V}_d (u_1,\ldots,u_d)\right|_{u_k =0} & = t_d,\qquad d \in \{3,\ldots,r + 1\}, \\
\left. \mathscr{P}(u_1,u_2)\right|_{u_k=0}& = 1\,.
\end{split} 
\end{equation*} 
This can be used to evaluate the weight $\mathscr{W}_{\rm cil}(M')$ of a multi-ciliated map  $M'$ at $u_i = 0$ --- recall that the $u$s are either equal to $\lambda_1$ or to the $z$s, which are in the present situation all set to zero. Thus, the local weight $t_d$ for a black vertex of degree $d$ in the multi-ciliated map $M'$ can be interpreted as a local weight for an internal face of degree $d$ in the dual fully-simple map $M$. Unlike $\mathscr{W}_{\rm cil}(M')$ in Definition~\ref{def:wcil}, the weight $\mathscr{W}(M)$ introduced in \eqref{def:Wscr} contains a factor
$$
\alpha^{2 - 2g(M) - \#\mathcal{V}(M)} = \alpha^{2 - 2g(M') - \#\mathcal{F}(M')} = \alpha^{-\deg M'}.
$$
We deduce that this specialisation retrieves the weights for the standard notion of maps, in the form
$$
\left. \alpha^{-\deg M'}\,\mathscr{W}_{{\rm cil}}(M')\right|_{u_k = 0} = \mathscr{W}(M).
$$
\end{proof}
\begin{lem}\label{lem:fs:cil:generating}
For $i \in \{1,\ldots,n\}$, we introduce $z_i$ as the formal power series in $\alpha^{-1}$ and $w_i$ uniquely determined by
\[
\begin{cases}
V'(z_i) = \frac{w_i}{\alpha}\,,\\
z_i = \frac{w_i}{\alpha} + \mathcal{O}(\alpha^{-2})\,.
\end{cases}
\]
Then, for any $g \geq 0$ and $n \geq 1$:
\begin{equation}\label{eq:fs:cil}
X_{g,n}(w_1,\ldots,w_n)= C_{g,n}(z_1,\ldots,z_n)|_{\lambda_1 = 0} + \delta_{g,0}\delta_{n,1} \,\alpha \left(\frac{1}{w_1}+ z_1\right).
\end{equation}
\hfill $\star$
\end{lem}
\begin{proof}
From the definition of $X_{g,n}$ and Equation~\eqref{eq:multi:fs}, we have
\begin{equation}\label{eq:proof:tr:fs}
\begin{split}
X_{g,n}(w_1,\ldots,w_n) &= \sum\limits_{k_1,\ldots, k_n \geq 1} w_1^{k_1-1}\cdots w_n^{k_n-1}\, \mathrm{FSMap}_{g;(k_1,\ldots,k_n)} + \delta_{g,0}\delta_{n,1} \frac{\alpha}{w_1}  \\
&= \sum\limits_{k_1,\ldots, k_n \geq 1} w_1^{k_1-1}\cdots w_n^{k_n-1} \,\left.S_{g;(k_1,\ldots,k_n)}\big([0^{k_1}],\ldots,[0^{k_n}]\big)\right|_{\lambda_1 = 0} + \delta_{g,0}\delta_{n,1} \frac{\alpha}{w_1}.
\end{split}
\end{equation}

Now, specifying all the $z_{i,j}=0$ in the Lemma~\ref{lem:multicil:to:cil:1} relating multi-ciliated and ciliated generating series, we obtain:
\begin{equation}\label{eq:proof:tr:fs:2}
\begin{split} 
& \quad \left.S_{g;(k_1,\ldots,k_n)}\big([0^{k_1}],\ldots,[0^{k_n}]\big)\right|_{\lambda_1 = 0} \\
&= \left.\frac{\alpha^{n-(k_1+\ldots+k_n)}}{(k_1-1)! \cdots (k_n-1)!} \frac{\partial^{k_1-1}}{\partial V'(z_1)^{k_1-1}}\cdots \frac{\partial^{k_n-1}}{\partial V'(z_n)^{k_n-1}} C_{g,n}(z_1,\ldots,z_n) \right|_{\substack{\lambda_1 = 0 \\ z_i = 0}} \\
& \quad + \delta_{g,0}\delta_{n,1} \frac{1}{\alpha^{k_1-1}} \left.\frac{\partial^{k_1-1}}{\partial V'(z_1)^{k_1-1}} \left(\alpha\, z_1 \right)\right|_{z_1=0}.
\end{split} 
\end{equation}

It is then natural to introduce the change of formal variable
$$
\alpha V'(z_i) = w_i\qquad {\rm such}\,\,{\rm that}\qquad  z_i = \frac{w_i}{\alpha} + \mathcal{O}(\alpha^{-2}),
$$
which is easily seen to exist and to be unique.

\medskip

\textbf{Case $(g,n)\neq (0,1)$.} Inserting Equation~\eqref{eq:proof:tr:fs:2} in the formula \eqref{eq:proof:tr:fs}, we recognise $X_{g,n}$ as a Taylor expansion of $C_{g,n}$ around 0:
\begin{equation*}
\begin{split}
& \quad X_{g,n}(w_1,\ldots,w_n) \\
&= \sum\limits_{k_1,\ldots, k_n \geq 1}  \left.\frac{w_1^{k_1-1}\cdots w_n^{k_n-1}}{(k_1-1)! \cdots (k_n-1)!} \frac{\partial^{k_1-1}}{\partial (\alpha V)'(z_1)^{k_1-1}}\cdots \frac{\partial^{k_n-1}}{\partial(\alpha V)'(z_n)^{k_n-1}} C_{g,n}(z_1,\ldots,z_n) \right|_{\substack{\lambda_1 = 0 \\ z_i = 0}} \\
&=\!\!\!\! \sum\limits_{k_1,\ldots,k_n\geq 0} \left.\left[\prod\limits_{i=1}^{n} \frac{w_i^{k_i}}{k_i!}  \frac{\partial^{k_i}}{\partial w_i^{k_i}}\right]C_{g,n}(z_1,\ldots,z_n) \right|_{\substack{\lambda_1 = 0 \\ z_i = 0}} \\
&= \left. C_{g,n}(z_1,\ldots,z_n)\right|_{\lambda_1 = 0}.
\end{split}
\end{equation*}
The equalities hold as formal series in $\alpha^{-1}$ and $w_1,\ldots,w_n$. 

\medskip

\textbf{Case of the disc.}  The only difference is the presence of the extra term in Equation~\eqref{eq:proof:tr:fs}, corresponding to the degenerate fully simple map in $\mathbf{M}_{0,1}(1)$, and the presence of an extra term in Lemma~\ref{lem:multicil:to:cil:1} relating ciliated and multi-ciliated generating series:
$$
S_{0,(k)}(z_{1,1},\ldots,z_{1,k}) = \frac{1}{\alpha^{k - 1}} \sum_{j = 1}^{k} \frac{C_{0,1}(z_{1,j}) + \alpha z_{1,j}}{\prod\limits_{\substack{l = 1 \\ l \neq j}}^{k} \big(V'(z_{1,j}) - V'(z_{1,l})\big)}.
$$
Similarly to the previous computation, we obtain:
\begin{equation*}
\begin{split}
X_{0,1}(w_1) &= \frac{\alpha}{w_1}+ \sum\limits_{k_1\geq 1} \frac{w_1^{k_1 - 1}}{(k_1 - 1)!}\left. \frac{\partial^{k_1 - 1}}{\partial (\alpha V)'(z_1)^{k_1- 1}}\left(C_{0,1}(z_1)+\alpha\, z_1\right) \right|_{\substack{\lambda_1 = 0 \\ z_1=0}} \\
&= \frac{\alpha}{w_1} + \left.C_{0,1}(z_1)\right|_{\lambda_1 = 0} + \alpha\, z_1.
\end{split}
\end{equation*}
which is an identity between formal Laurent series in $\alpha^{-1}$ and formal power series in $w_1$.  
\end{proof}

\br
\label{rem:addu}The additional term of Lemma~\ref{lem:multicil:to:cil} for the case $(g,n,\underline{k})=(0,1,(2))$ comes from the term 
$$
\left.\mathscr{P}(z_{1,1},z_{1,2})\right|_{z_{1,1}= z_{1,2} = 0}=\mathscr{P}(0,0)=1,
$$ 
which corresponds to the special multi-ciliated map, dual to the degenerate fully simple map in $\mathbf{M}_{0,1}(2)$:
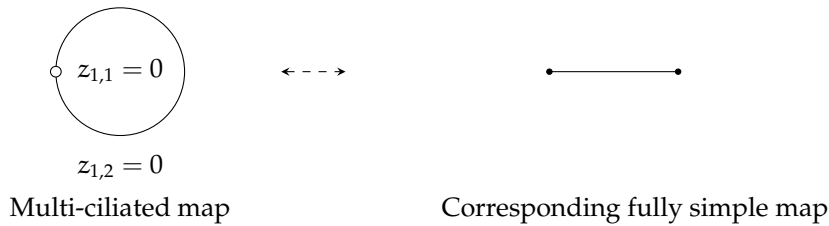
\begin{center}
\begin{tikzpicture}[ipe stylesheet] \useasboundingbox (48, 736) rectangle (396, 832); \node at (112, 800) {$z_{1,1}=0$}; \draw (112, 800) circle[radius=24]; \filldraw[fill=white] (88, 800) circle[radius=2]; \node at (112, 764) {$z_{1,2}=0$}; \node at (112, 748) {Multi-ciliated map}; \node at (304, 748) {Corresponding fully simple map}; \draw (272, 800) -- (320, 800); \filldraw[fill=black] (320, 800) circle[radius=1]; \filldraw[fill=black] (272, 800) circle[radius=1]; \draw[stealth-stealth, dashed] (172, 800) -- (196, 800); \end{tikzpicture}
\end{center}
This special case yields the additional term of Lemma~\ref{lem:multicil:to:cil:1} and hence of \eqref{eq:proof:tr:fs:2}.
\hfill $\star$
\er

For $g \geq 0$ and $n \geq 1$, we define the specialisation of the $n$-differential \eqref{gammaTR} to $\Lambda = \{0\}$:
$$ 
\chi_{g,n}(\zeta_1,\ldots,\zeta_n) \coloneqq \left.\Gamma_{g,n}(\zeta_1,\ldots,\zeta_n)\right|_{\lambda_1 = 0}.
$$
As a consequence of Lemma~\ref{lem:fs:cil:generating}, this is also:
\begin{equation}
\label{alphan}\begin{split}
\chi_{g,n}(\zeta_1,\ldots,\zeta_n) & = \bigg(X_{g,n}(w_1,\ldots,w_n) + \frac{\delta_{g,0}\delta_{n,2}}{\big(x(\zeta_1) - x(\zeta_2)\big)^2}\bigg)\dd x(\zeta_1)\cdots \dd x(\zeta_n) \\
&  = \bigg(\alpha^{-n}\,X_{g,n}(w_1,\ldots,w_n) + \frac{\delta_{g,0}\delta_{n,2}}{\big(w(\zeta_1) - w(\zeta_2)\big)^2} \bigg)\dd w_1 \cdots \dd w_n
\end{split}
\end{equation}
because we have $x(\zeta_i) = Q(\zeta_i) = V'(z_i) = \frac{w_i}{\alpha}$ as specified in the Lemma. The specialisation of Theorem~\ref{thm:tr:ciliated} now implies our main result, \textit{i.e.} that the fully simple generating series satisfies the topological recursion.

\begin{thm}
\label{mainprop} For (generic) values of $t_3,\ldots,t_{r + 1}$ such that the ramification points are simple, the $n$-differentials $\chi_{g,n}$ can be analytically continued to meromorphic $n$-forms on the spectral curve of Definition~\ref{def:spec:curve} specialised to  $\Lambda = \{0\}$. The analytic continuations, still denoted $\chi_{g,n}$, satisfy the topological recursion on this spectral curve (see \eqref{gammaTR}). The generating series $X_{g,n}$ are retrieved by expansion when $w_i \rightarrow 0$. \hfill $\star$
\end{thm}

\subsection{Identification of the spectral curve}
\label{S43} 
The purpose of this section is to give a direct identification of the spectral curve based on the previous sections. 

\begin{thm} \label{maincor}The fully simple spectral curve is obtained from the ordinary spectral curve by exchanging the role of $\mathsf{x}$ and $\mathsf{y}$, and the conjecture of \cite{BG-F18} holds. \hfill $\star$
\end{thm}

\begin{rem}
The topological recursion in Theorem~\ref{thm:tr:ciliated} and Theorem~\ref{mainprop} is stated for (generic) parameters, so that the corresponding spectral curve has only simple ramification points. This assumption can be waived using the Bouchard--Eynard topological recursion \cite{BHLMR14,BoEy13}, which allows spectral curves with ramification points of higher order, and exploiting its continuity properties with respect to parameters --- see \textit{e.g.} \cite[Section 2.2]{BDKLM20} for an example of such an argument. \hfill $\star$
\end{rem}

\begin{proof}[Proof of Theorem~\ref{maincor}]
We start by simplifying the description of the spectral curve for fully simple maps, which will eventually make the comparison to the spectral curve for ordinary maps clear. Following Theorem~\ref{mainprop} we  have to consider the specialisation of Definition~\ref{def:spec:curve} to $\Lambda = \{0\}$. This is the spectral curve in parametrised form:
\begin{equation*}
\begin{cases}
x(\zeta) = Q(\zeta),\vspace{0.1cm} \\
y(\zeta) = \alpha\, \zeta +\frac{1}{Q'(a)(\zeta-a)}, \vspace{0.1cm} \\
\chi_{0,1}(\zeta) = y(\zeta)\dd x(\zeta), \vspace{0.1cm} \\
\chi_{0,2}(\zeta_1,\zeta_2) = \frac{\dd \zeta_1 \dd \zeta_2}{(\zeta_1 - \zeta_2)^2},
\end{cases}
\end{equation*}
where $a \in \mathbb{C}[[\alpha^{-1}]]$ and the polynomial $Q$ with coefficients in $\mathbb{C}[[\alpha^{-1}]]$ are determined by:
\begin{equation}
\label{caracorum}\begin{array}{lll}
Q(\zeta) \mathop{=}\limits_{\zeta \rightarrow \infty} V'\left(\zeta+ \frac{1}{\alpha}\frac{1}{Q'(a)(\zeta-a)}\right) + \mathcal{O}(\zeta^{-1}), & \qquad & Q(\zeta) =V'(\zeta) + \mathcal{O}(\alpha^{-1}), \vspace{0.15cm} \\
Q(a) =V'(0)=0, & \qquad & a = \mathcal{O}(\alpha^{-1}).
\end{array}
\end{equation}
Besides, the variables $w \in \alpha\mathbb{C}(\zeta)[[\alpha^{-1}]]$ (resp.~$z \in \mathbb{C}(\zeta)[[\alpha^{-1}]]$) that should be used to extract the fully simple (resp.~ciliated) generating series are determined by:
\[ \frac{w}{\alpha} = V'(z) = Q(\zeta),\qquad \frac{w}{\alpha} = z + \mathcal{O}(\alpha^{-2}) = \zeta + \mathcal{O}(\alpha^{-2}).\]

Let now $\hat{\chi}_{g,n}$ be the multi-differentials computed by the topological recursion on the rescaled spectral curve:
\begin{equation*}
\begin{cases}
\hat{x}(\zeta)  = x(\zeta), \vspace{0.1cm} \\
\hat{y}(\zeta)  = \frac{y(\zeta)}{\alpha} = \zeta +\frac{1}{\alpha}\frac{1}{Q'(a)(\zeta-a)}, \vspace{0.1cm} \\
\hat{\chi}_{0,1}(\zeta)  = \hat{y}(\zeta)\dd\hat{x}(\zeta), \vspace{0.1cm} \\
\hat{\chi}_{0,2}(\zeta_1,\zeta_2) = \chi_{0,2}(\zeta_1,\zeta_2).
\end{cases}
\end{equation*}
As the only effect of this rescaling on the topological recursion is to multiply the recursion kernel by $\alpha$ (compare with \eqref{gammaTR1}), and the topology $(g,n)$  is reached after $2g - 2 +n$ steps of the recursion, we have:
\[\hat{\chi}_{g,n}(\zeta_1,\ldots,\zeta_n) = \alpha^{2g - 2 +n}\chi_{g,n}(\zeta_1,\ldots,\zeta_n)\,.\] 
Taking into account the $\alpha^{-n}$ present in \eqref{alphan} and coming back to Definition~\ref{def:gen:fun:ord:fs} for $X_{g,n}$, we see that $\hat{\chi}_{g,n}$ are generating series of fully simple maps with modified weights: 
\[
\hat{\mathscr{W}}(M)=\frac{\alpha^{-\#\mathcal{V}(M)}}{\# \Aut(M)}t_{3}^{f_3(M)}\cdots t_{r+1}^{f_{r+1}(M)},
\]
This choice for the weight is the one made for the enumeration of ordinary maps \textit{e.g.} in \cite[Chapter 3]{Eyn16}. More precisely, for any $g \geq 0$ and $n \geq 1$ we have:
\begin{equation}
\label{TRfull} \hat{\chi}_{g,n}(\zeta_1,\ldots,\zeta_n) = \sum_{\substack{M \in \mathbf{M}_{g,n} \\ {\rm fully}\,\,{\rm simple}}} \hat{\mathscr{W}}(M)\,w_1^{\deg(\partial_1M) - 1} \cdots w_n^{\deg(\partial_n M) - 1} \dd w_1 \cdots \dd w_n.
\end{equation}
where $w_i = \alpha Q(\zeta_i)$ as specified in Lemma~\ref{lem:fs:cil:generating}, and the equation should be understood as the equality of the all-order series expansion of the left-hand side when $w_i \rightarrow 0$ with the formal series on the right-hand side.

\medskip

Now, we introduce a different uniformising coordinate in $\mathbb{P}^1$, which we call $\theta$ and is related to $\zeta$ by:
\begin{equation}
\label{birat}
\zeta(\theta) = a + c\theta^{-1},\qquad c:= \big(\alpha\, Q'(a)\big)^{-\frac{1}{2}}.
\end{equation}
It can be easily checked that $c \in \mathbb{C}[[\alpha^{-\frac{1}{2}}]]$, in particular
\begin{equation}
\label{calva}c = \mathcal{O}(\alpha^{-\frac{1}{2}}).
\end{equation}
We then find
\begin{equation}
\label{fullysimplesp} \check{\mathcal{S}}\,:\,\qquad \begin{cases}
\hat{x}(\zeta(\theta)) = Q\big(a + c\theta^{-1}\big)\,, \vspace{0.1cm} \\
 \hat{y}(\zeta(\theta)) = a + c(\theta+ \theta^{-1})\,, \vspace{0.1cm} \\
\hat{\chi}_{0,2}(\zeta(\theta_1),\zeta(\theta_2)) = \frac{\dd \theta_1 \dd \theta_2}{(\theta_1 - \theta_2)^2}\,,
\end{cases}
\end{equation}
and the polynomial $Q$ is characterised by
$$
Q\big(a + c\theta^{-1}\big) \mathop{=}\limits_{\theta \rightarrow 0} V'\big(a + c(\theta + \theta^{-1})\big) + \mathcal{O}(\theta).
$$
In other words:
$$
Q\big(a +c\theta^{-1}\big) = \Big[V'\big(a + c(\theta + \theta^{-1})\big)\Big]_{\leq 0}\;,
$$
where $[\cdots]_{\leq 0}$ is the polynomial part in the variable $\theta^{-1}$. Besides, we have the constraints
\begin{equation}
\label{cQ}Q(a) = 0\qquad {\rm and}\qquad  \lim_{\theta \rightarrow \infty} \theta\,Q(a + c\theta^{-1}) = cQ'(a) = (\alpha c)^{-1}.
\end{equation}
The first one is a reminder from \eqref{caracorum} and the second one follows from the definition of $c$ in \eqref{birat}. For comparison, the spectral curve for ordinary maps is --- see \textit{e.g.}\footnote{In \cite{Eyn16}, the triple $(\alpha^{-1},a,c)$ was denoted $(t,\alpha,\gamma)$.} \cite[Section 3.1.3]{Eyn16}:
\begin{equation}
\label{spnormal}
\mathcal{S}\,:\,\qquad \begin{cases}
\mathsf{x}(\theta)  = a + c(\theta + \theta^{-1})\,, \vspace{0.1cm} \\
\mathsf{y}(\theta)  = \big[V'(a + c(\theta + \theta^{-1}))\big]_{\leq 0}\,, \vspace{0.1cm} \\
\omega_{0,1}(\theta)  = \mathsf{y}(\theta)\dd \mathsf{x}(\theta)\,, \vspace{0.1cm} \\
\omega_{0,2}(\theta_1,\theta_2)  = \frac{\dd \theta_1 \dd \theta_2}{(\theta_1 - \theta_2)^2}\,,
\end{cases}
\end{equation}
where $c$ (up to a sign) and $a$ are uniquely determined by the conditions
\begin{equation}
\label{xycond}\begin{split}
& \mathsf{y}(\theta) \mathop{\sim}_{\theta \rightarrow \infty} \frac{\alpha^{-1}}{x(\theta)} \sim \frac{1}{\alpha c \theta}\,, \\
& c  = \mathcal{O}(\alpha^{-\frac{1}{2}})\,,  \\
& a = \mathcal{O}(\alpha^{-1}) \,,
\end{split}
\end{equation}
and $\alpha^{-1}$ is the weight per vertex. We recognise
$$
\mathsf{y}(\theta) = \hat{x}(\theta),\qquad \mathsf{x}(\theta) = \hat{y}(\theta),\qquad \omega_{0,2}(\theta_1,\theta_2) = \hat{\chi}_{0,2}(\theta_1,\theta_2),
$$
with parameters $(a,c)$ determined in an identical way: the first condition of \eqref{xycond} is equivalent to \eqref{cQ}, the second condition is equivalent to \eqref{calva} (the sign ambiguity amounts to the choice of squareroot), and the third condition matches the last condition in \eqref{caracorum}.

\medskip

It is well-known that topological recursion on the spectral curve \eqref{spnormal} computes the generating series of ordinary maps. More precisely, for $g \geq 0$ and $n \geq 1$, let us define:
\begin{equation}
\label{TRord}\omega_{g,n}(\theta_1,\ldots,\theta_n) \coloneqq \bigg(W_{g,n}(\mathsf{x}(\theta_1),\ldots,\mathsf{x}(\theta_n)) + \frac{\delta_{g,0}\delta_{n,2}}{\big(\mathsf{x}(\theta_1) - \mathsf{x}(\theta_2)\big)^2}\bigg)\dd \mathsf{x}(\theta_1) \cdots \dd \mathsf{x}(\theta_n).
\end{equation}
It is established in \cite{Eyn04,Eyn16}  that for any $t_3,\ldots,t_{r + 1} \in \mathbb{C}$, the multi-differentials $\omega_{g,n}$ can be ana\-ly\-ti\-cal\-ly continued to meromorphic $n$-forms on the spectral curve \eqref{spnormal}. If we still denote $\omega_{g,n}$ the analytic continuations, $\omega_{0,2}$ is the standard bidifferential and $\omega_{g,n}$ for $2g - 2 + n > 0$ is computed by the topological recursion on this spectral curve. This explains the formulation of the claim, and concludes the proof, given Theorem~\ref{mainprop}. 
\end{proof}

\section{Applications}

\label{S5}

\label{S5}

In this section, we briefly explain motivation for this work, and some consequences of Theorems~\ref{mainprop}-\ref{maincor} which pave the way for future investigations.

\medskip

\subsection{Symplectic invariance}

\subsubsection{Context}

Let $\mathcal{S} = (\mathcal{C},x,y,\omega_{0,2})$ be a spectral curve and $P$ the set of zeroes of $\dd x$, the topological recursion constructs\footnote{There are assumptions on the spectral curve for this construction to be well-defined, we refer to \cite{BKS} for a discussion and the weakest currently known set of assumptions.} multi-differentials $\omega_{g,n}$ indexed by $g \geq 0$ and $n \geq 1$, but also the following numbers --- called \emph{free energies} --- indexed by $g \geq 2$:
\begin{equation}
\label{defFg} \mathfrak{F}_g[\mathcal{S}] = \frac{1}{2 - 2g} \sum_{\rho \in P} \Res_{z = \rho} \bigg(\int_{o_{\rho}}^z y\dd x\bigg)\omega_{g,1}(z).
\end{equation}
Here, $o_{\rho} \in \mathcal{C}$ is an arbitrary point in a small contractible neighborhood of $\rho$ and we integrate from $o_{\rho}$ to $z$ in such a neighborhood.
\begin{rem}
If $y\dd x$ is meromorphic on a connected curve $\mathcal{C}$, we can also choose $o_{\rho}$ independent of $\rho$, and \eqref{defFg} does not depend on the path of integration from $o$ to $z$ since $\Res_{z = \rho} \omega_{g,1} = 0$ for any $\rho \in P$, see \textit{e.g.} \cite{EO07inv}. \hfill $\star$
\end{rem}
 Let $\check{\mathcal{S}} = (\mathcal{C},y,x,\omega_{0,2})$ be the spectral curve where the role of $x$ and $y$ are exchanged, and $\check{P}$ be the set of zeroes of $\dd y$. It is expected that for reasonable spectral curves, we have the equality
\begin{equation}
\label{Fginv}
\mathfrak{F}_{g}[\mathcal{S}] = \mathfrak{F}_{g}[\check{\mathcal{S}}].
\end{equation}
although the multi-differentials constructed by the topological recursion for $\mathcal{S}$ and $\check{\mathcal{S}}$ are different. This property is called symplectic invariance. It is deep and still mysterious. In applications of topological recursion in Gromov--Witten theory of toric Calabi--Yau threefolds \cite{BKMP09,EO15}, it corresponds for instance to the framing invariance of the closed sector. The precise meaning of ''reasonable'', \textit{i.e.} the minimal assumptions on the spectral curve under which \eqref{Fginv} is expected to hold (perhaps after  adding certain explicit terms on the right-hand side) are not known. Eynard and Orantin have proposed in \cite{EO2MM,EOxy} a rather involved derivation via the two-matrix model; yet, their result does not seem to always apply in cases of interest. Understanding better the origin of symplectic invariance, formulating it precisely and obtaining its proof under the weakest possible assumptions remains admittedly a fundamental and open problem in the theory of topological recursion.

\medskip

Here, we are in position to give an interpretation of some pieces of the puzzle when $\mathcal{S} = (\mathbb{P}^1,\mathsf{x},\mathsf{y},\omega_{0,2})$ is the spectral curve \eqref{spnormal} governing ordinary maps. In that case $\check{\mathcal{S}} = (\mathbb{P}^1,\mathsf{y},\mathsf{x},\hat{\omega}_{0,2})$ is the spectral curve \eqref{fullysimplesp} which we proved to govern fully simple maps, and numerically, \eqref{Fginv} does not seem to hold as such. 

\subsubsection{Free energy computations}

\medskip

We shall compute both sides of \eqref{Fginv} for our two spectral curves, in terms of ordinary and fully simple generating series in topology $(g,1)$. We use the letters $\omega$ (resp.~$\check{\omega}$) for the multi-differentials for the spectral curve $\mathcal{S}$ (resp.~$\check{\mathcal{S}}$), so in fact $\check{\omega}$ coincide with $\hat{\chi}$ of Section~\ref{S43}. We denote $P = \{-1,1\}$ the set of zeroes of $\mathsf{x}'(\theta) = c(1 - \theta^{-2})$ and $\check{P}$ the zeroes of $\mathsf{y}'(\theta)$. By a continuity argument, it is sufficient to prove the result for $t_3,\ldots,t_{r + 1}$ such that the zeroes of $\mathsf{y}'$ are simple, \textit{i.e.} $\# \check{P} = r$.  Notice that
\begin{equation}
\label{Yves}\mathsf{y}(\theta) = \big[V'(\mathsf{x}(\theta))\big]_{\leq 0}  = V'(\mathsf{x}(\theta)) + \mathcal{O}(\theta),
\end{equation}
where the $\mathcal{O}(\theta)$ is in fact a polynomial in $\theta$.

\medskip

Let us fix $g \geq 2$. According to the basic properties of the topological recursion \cite{EO07inv}, $\omega_{g,1}$ (resp. $\check{\omega}_{g,1}$) is a meromorphic $1$-form with poles at $P$ (resp.~$\check{P}$) and zero residues. Then, we can introduce the rational functions
$$
\Phi_{g,1}(\theta) = \int_{\infty}^{\theta} \omega_{g,1},\qquad \check{\Phi}_{g,1}(\theta) = \int_{\infty}^{\theta} \check{\omega}_{g,1}.
$$
Another basic property is the linear loop equation (see \textit{e.g.} \cite{BSblob}), which states that
\begin{equation}
\label{linin}
\sum_{\tilde{\theta} \in \mathsf{x}^{-1}(\mathsf{x}(\theta))} \omega_{g,1}(\tilde{\theta})
\end{equation}
is holomorphic near $P$. But here $\mathsf{x}^{-1}(\mathsf{x}(\theta)) = \{\theta,\theta^{-1}\}$; in particular, the involution $\theta \mapsto \theta^{-1}$ giving the second point in this fiber is globally defined on $\mathbb{P}^1$. Therefore, the left-hand side of \eqref{linin} is a holomorphic $1$-form on $\mathbb{P}^1$, hence
\begin{equation}
\label{omeg1line}\omega_{g,1}(\theta) + \omega_{g,1}(\theta^{-1}) = 0.
\end{equation}
See also \cite{Eyn16}. Note that the linear loop equation for $\check{\omega}_{g,1}$ states that $\sum_{\tilde{\theta} \in \mathsf{y}^{-1}(\mathsf{y}(\theta))} \check{\omega}_{g,1}(\tilde{\theta})$ is holomorphic near $\check{P}$, which does not lead to any formula for $\check{\omega}_{g,1}(\theta^{-1})$. We also mention the property obtained in \cite{EO07inv}:
\begin{equation}
\label{xypr}\sum_{\rho \in \check{P}} \Res_{\theta = \rho} \mathsf{x}(\theta)\mathsf{y}(\theta) \check{\omega}_{g,1}(\theta) = 0.
\end{equation}
An analogous one is also true for $\omega_{g,1}$ but we will not need it.

\medskip

We recall that topological recursion for ordinary maps (see \eqref{TRord}) and for fully simple maps (see~\eqref{TRfull}) yields the expansions:
\begin{equation}
\label{expnunu}\begin{split}
\omega_{g,1}(\theta) & \mathop{=}_{\theta \rightarrow \infty}  \,\,\sum_{k = 1}^{r + 1} {\rm Map}_{g;(k)}\,\frac{\dd \mathsf{x}(\theta)}{x(\theta)^{k + 1}} + \mathcal{O}\bigg(\frac{\dd \mathsf{x}(\theta)}{\mathsf{x}(\theta)^{r + 3}}\bigg), \\
\check{\omega}_{g,1}(\theta) & \mathop{=}_{\theta \rightarrow \infty} \,\, \sum_{k = 1}^{r + 1} {\rm FSMap}_{g;(k)}\,\mathsf{y}(\theta)^{k - 1}\dd \mathsf{y}(\theta) + \mathcal{O}\big(\mathsf{y}(\theta)^{r + 1}\dd \mathsf{y}(\theta)\big) \\
& \mathop{=}_{\theta \rightarrow \infty} \,\,- \sum_{k = 1}^{r + 1} \sum_{\substack{\ell_1,\ldots,\ell_{k} \geq 0 \\ \ell_1 + \cdots + \ell_k + k \leq r + 1}} \frac{(\ell_1 + \cdots + \ell_k + k)\,{\rm FSMap}_{g;(k)}}{k\,\mathsf{x}(\theta)^{\ell_1 + \cdots + \ell_k + k + 1}}\bigg[\prod_{i = 1}^{k} {\rm Map}_{0,(\ell_i)}\bigg]\dd \mathsf{x}(\theta) \\
&\quad\quad + \mathcal{O}\bigg(\frac{\dd\mathsf{x}(\theta)}{\mathsf{x}(\theta)^{r + 3}}\bigg).
\end{split}  
\end{equation} 
We could truncate the sum in the last line using $\mathsf{y}(\theta) = \mathcal{O}\big(\mathsf{x}(\theta)^{-1}\big)$ when $\theta \rightarrow \infty$. We will see that the expansion of $\check{\omega}_{g,1}(\theta)$ near $\theta = 0$ plays a role for the computation of $\mathfrak{F}_g[\check{\mathcal{S}}]$. We therefore introduce a name for its coefficients:
$$
\check{\omega}_{g,1}(\theta) \mathop{=}_{\theta \rightarrow 0}\,\, \sum_{k = 1}^{r + 1} {\rm Rest}_{g,(k)}\,\frac{\dd \mathsf{x}(\theta)}{\mathsf{x}(\theta)^{k + 1}} + \mathcal{O}\bigg(\frac{\dd\mathsf{x}(\theta)}{\mathsf{x}(\theta)^{r + 3}}\bigg).
$$

\medskip

We are ready to compute $\mathfrak{F}_g[\mathcal{S}]$, starting from \eqref{defFg}.

\begin{lem}
\label{lemFgs} For $g \geq 2$, the generating series of closed maps of genus $g$ satisfy:
$$
(2 - 2g)\mathfrak{F}_g[\mathcal{S}] = \alpha \partial_{\alpha} {\rm Map}_{g,\emptyset} =  - \frac{{\rm Map}_{g;(2)}}{2} + \sum_{k = 3}^{r + 1} t_k \frac{{\rm Map}_{g;(k)}}{k}.
$$
\hfill $\star$
\end{lem}
\begin{proof} Integration by parts in \eqref{defFg} yields:
$$
(2 - 2g)\mathfrak{F}_g[\mathcal{S}] = \sum_{\rho \in P} \Res_{\theta = \rho} \bigg(\int_{\infty}^{\theta} \mathsf{y}\dd \mathsf{x}\bigg) \omega_{g,1}(\theta) = - \sum_{\rho \in P} \Res_{\theta = \rho}  \Phi_{g,1}(\theta)\,\mathsf{y}(\theta) \dd \mathsf{x}(\theta).
$$
The $1$-form $\omega_{0,1}(\theta) = \mathsf{y}(\theta)\dd\mathsf{x}(\theta)$ has a simple pole at $\theta = \infty$ and a pole of order $r + 2$ at $\theta = 0$. Besides, the function $\Phi_{g,1}(\theta)$ has a simple zero at $\theta = \infty$. Moving contours we deduce that
$$
(2 - 2g)\mathfrak{F}_g[\mathcal{S}] = \Res_{\theta = 0} \Phi_{g,1}(\theta) \,\mathsf{y}(\theta)\dd\mathsf{x}(\theta)  = \Res_{\theta = 0}\Phi_{g,1}(\theta)\, \dd V(\mathsf{x}(\theta)) ,
$$
where we have used \eqref{Yves}. We then perform the change of variable $\theta \mapsto \theta^{-1}$ and use that $\mathsf{x}$ is invariant while $\Phi_{g,1}$ is antiinvariant (by integration of \eqref{omeg1line}) to find
$$
(2 - 2g)\mathfrak{F}_g[\mathcal{S}] = -\Res_{\theta = \infty} \Phi_{g,1}(\theta)\,\dd V(\mathsf{x}(\theta)) = \Res_{\theta = \infty} V(\mathsf{x}(\theta)) \omega_{g,1}(\theta).
$$ 
We rather use the local coordinate $\mathsf{x}$ near $\theta = \infty$ and insert the expansion \eqref{expnunu} and Definition~\ref{def:potential} of the potential, which results in:
\begin{equation}
\label{Fgun}\begin{split}
(2 - 2g)\mathfrak{F}_g[\mathcal{S}] & = \Res_{x = \infty} \dd x\bigg(\frac{x^2}{2} - \sum_{m = 3}^{r + 1} \frac{t_m}{m}\,x^m\bigg)\,\bigg(\sum_{k = 1}^{r + 1} {\rm Map}_{g;(k)}\,\frac{\dd x}{x^{k + 1}}\bigg) \\
& = - \frac{{\rm Map}_{g;(2)}}{2} + \sum_{k = 3}^{r + 1} t_k \frac{{\rm Map}_{g;(k)}}{k}.
\end{split}
\end{equation}
The weight of an ordinary map $M$ includes a factor $\alpha^{\deg M} = \alpha^{2 - 2g(M) - \#\mathcal{V}(M)} = \alpha^{-\#\mathcal{E}(M) + \#\mathcal{F}(M)}$. Closed ordinary maps of genus $g$ with a marked (non-oriented) edge are in bijection with ordinary maps of genus $g$ with an (unrooted) boundary face of degree $2$: just glue the two edges of the boundary face. Closed ordinary maps of genus $g$ with a marked (unrooted) face are in bijection with ordinary maps of genus $g$ with an (unrooted) boundary face of degree $k \in \{3,\ldots,r + 1\}$. All together, these observations imply that
$$
\alpha \partial_{\alpha} {\rm Map}_{g,\emptyset} = - \frac{{\rm Map}_{g;(2)}}{2} + \sum_{k = 3}^{r + 1} t_k \frac{{\rm Map}_{g;(k)}}{k} = (2 - 2g)\mathfrak{F}_{g}[\mathcal{S}].
$$
\end{proof}

\medskip

We now turn to the free energy for $\check{\mathcal{S}}$. It is not directly expressed in terms of ${\rm FSMap}$.
\begin{lem}
For $g \geq 2$, we have
\begin{equation}
\label{idngun}(2-2g)\mathfrak{F}_g[\check{\mathcal{S}}] =  - \frac{{\rm Rest}_{g,(2)}}{2} + \sum_{k = 3}^{r + 1} \frac{{\rm Rest}_{g,(k)}}{k}.
\end{equation}
\hfill $\star$
\end{lem}
\begin{proof} Writing $\mathsf{x}\dd \mathsf{y} = -\mathsf{y}\dd \mathsf{x} + \dd(\mathsf{x}\mathsf{y})$ and with the help of \eqref{xypr}, we compute:
\begin{equation}
\label{Fgautre}\begin{split}
(2 - 2g)\mathfrak{F}_g[\check{\mathcal{S}}] & = \sum_{\rho \in \check{P}} \Res_{\theta = \rho} \bigg(\int_{o_{\rho}}^{\theta} \mathsf{x}\dd \mathsf{y}\bigg) \check{\omega}_{g,1}(\theta) = -\sum_{\rho \in \check{P}} \Res_{\theta = \rho} \bigg(\int_{o_{\rho}}^{\theta} \mathsf{y}\dd \mathsf{x}\bigg) \check{\omega}_{g,1}(\theta)  \\
& = \sum_{\rho \in \check{P}} \Res_{\theta = \rho} \check{\Phi}_{g,1}(\theta)\,\mathsf{y}(\theta)\dd\mathsf{x}(\theta) = - \Res_{\theta = 0,\infty} \check{\Phi}_{g,1}(\theta)\,\mathsf{y}(\theta)\dd\mathsf{x}(\theta).
\end{split}
\end{equation}
When $\theta \rightarrow \infty$, we have
$$
\mathsf{y}(\theta)\dd\mathsf{x}(\theta) = \mathcal{O}(\theta^{-1} \dd \theta),\qquad \check{\Phi}_{g,1}(\theta) = \mathcal{O}(\theta^{-1}),
$$
therefore $\theta = \infty$ does not contribute to the residue, and we find:
\begin{equation}
\begin{split}
(2 - 2g)\mathfrak{F}_g[\check{\mathcal{S}}] & = - \Res_{\theta = 0} \check{\Phi}_{g,1}(\theta)\,\dd V(\mathsf{x}(\theta)) \\
& =  \Res_{\theta = 0} V(\mathsf{x}(\theta))\,\check{\omega}_{g,1}(\theta) \\
& = \Res_{x = \infty} \bigg(\frac{x^2}{2} - \sum_{m = 3}^{r + 1} t_m\,\frac{x^m}{m}\bigg) \bigg(\sum_{k = 1}^{r + 1} {\rm Rest}_{g;(k)}\,\frac{\dd x}{x^{k + 1}}\bigg) \\
& = - \frac{{\rm Rest}_{g,(2)}}{2} + \sum_{k = 3}^{r + 1} \frac{{\rm Rest}_{g,(k)}}{k}.
\end{split}
\end{equation}
\end{proof}

\subsubsection{Commentary}

Lemma~\ref{lemFgs} related the enumeration of closed maps of genus $g$ to the enumeration of ordinary maps of genus $g$ with $1$ boundary face. One may try to apply a simplification procedure to this boundary face, so as to relate it further to a fully simple enumeration. However, in doing so, many topologies lower than $(g,1)$ may appear. To understand if symplectic invariance is true (or true up to additional terms), we would need to find a combinatorial interpretation of the generating series ${\rm Rest}_{g,(k)}$, stored in the $\theta \rightarrow 0$ series expansion of $\check{\omega}_{g,1}$. The fully simple enumeration itself is stored in the expansion at $\theta = \infty$. The particular form it takes in \eqref{expnunu} has a clear combinatorial meaning: one can attach  ordinary disks at each vertex of a simple boundary face to make it ordinary. However, an ordinary face can be obtained from a simple face in different ways as well, that would involve maps of lower topologies. Therefore, the combinatorial meaning of \eqref{idngun} is at present not elucidated although we expect there should be one.

\subsubsection{Relation to matrix model with external field}

Theorem~\ref{maincor} had received a conditional proof in \cite{BG-F18}, provided a milder version of symplectic invariance was true for the topological recursion for the matrix model with external field, from a combinatorial interpretation of the partition function of the matrix model with external field as a generating series of fully simple maps.
The definition of ciliated maps from \cite{BCEG-F} was also motivated by the matrix model with external field, seen as a generalisation of the Kontsevich matrix model, in order to study the so-called $r$-spin intersection numbers \cite{Witt93,FSZ10} on the moduli space of curves, as a generalisation of the study initiated by Kontsevich for $r=2$ \cite{Kon92} in which he gave a proof of Witten's conjecture \cite{Witt90}. This suggests that the concrete combinatorial tools developed around fully simple maps in relation to ordinary maps can be employed within the full generality of the matrix model with external field (without specifying the parameters and variables to $0$), which involves an even larger family of spectral curves.

\subsection{Enumeration of fully simple maps}

The enumeration of fully simple maps of genus $0$ was explicitly solved by Krikun \cite{Krikun} for triangulations (only $t_3\neq 0$). His method was later generalised by Bernardi--Fusy for planar quadrangulations (only $t_4\neq 0$) and boundary faces of even degrees. Using the predictions coming from the conjectural topological recursion those formulas were conjecturally generalised for any boundary face degrees \cite[Conjecture 1.9]{BG-F18}. That closed formula is now proved for $n\leq 4$ from a straightforward application of two steps of the topological recursion. The enumeration of discs and cylinders for any structure of internal faces was already established in \cite{BG-F18}.

\medskip

In general, possibly disconnected ordinary maps are known to be related to possibly disconnected fully simple maps via monotone Hurwitz numbers. This relation was established using Weingarten calculus in \cite{BG-F18} and via bijective combinatorics in \cite{BCDG-F19}. These formulas allow to compute the number of fully simple maps with certain constraints, if one is already able to compute the number of ordinary maps and (strictly or weakly) monotone Hurwitz numbers. The enumeration of connected maps in terms of the disconnected ones is possible making use of inclusion-exclusion formulas. The advantage of Theorem~\ref{maincor} is to solve directly the enumeration of fully simple maps for any topology, recursively on $2g-2+n$, and for any structure of the internal faces. For instance, it implies the enumeration of quadrangulations of topology $(1,1)$:
\begin{cor}
For $m\in\mathbb{Z}_{\geq 0}$, let $\phi_m=c^{2m}\frac{1+(m-1)\sqrt{1-12t_4}}{1-12t_4}$, where we $c^2= \frac{1 - \sqrt{1 - 12t_4}}{6t_4}.$
Then, 
\begin{align}\label{formulasSeries1}
\left.\mathrm{Map}_{1;(2(m+1))}\right|_{\alpha=1} & = \frac{(2m+1)!}{6\,m!^2}\,\phi_m, \ \text{ for } m\geq 0,\\
\left.\mathrm{FSMap}_{1;(2m)}\right|_{\alpha=1} & = \frac{(3m)!\, t_4^{m+1}}{4\,m!(2m-1)!}\,\phi_{3m+1}, \ \text{ for } m\geq 1.
\end{align}
\hfill $\star$
\end{cor}
The details of the proof and how to extract closed formulas from this corollary are detailed in \cite[Section 5.2.3]{BG-F18}.

\subsection{Functional relations and connection to free probability}

In free probability theory, the notion of independence of classical probabilities is replaced by a notion of freedom, which is particularly well adapted to study non-commutative probability spaces. Free cumulants are crucial objects that allow to characterise freedom in a simple way. Random matrices in the large size limit constitute an important class of free random variables. In \cite{MingoSpeicher, MMS07, speicher2007} a notion of higher order freeness was introduced to study these questions more finely. While first order free cumulants are defined in terms of moments using non-crossing partitions, the definition of higher order free cumulants involve intricate combinatorial objects, called non-crossing partitioned permutations.

\medskip

For $n=1$, the $R$-transform machinery \cite{Voiculescu2} gives a relation between the generating series of moments and of free cumulants, by functional inversion. For $n=2$, a functional relation between the ordinary and the free generating series was also found \cite{speicher2007}, already in a quite complicated way. Similar functional relations are unfortunately not known for $n\geq 3$, which leaves us with a rather complicated theory to compute with. In \cite[Section 11.2]{BG-F18}, for an arbitrary (formal) unitarily-invariant measure of the space of Hermitian matrices, the classical identification of moments of products of traces of Hermitian matrices with generating series of ordinary maps \cite{BIPZ} was extended to an identification of free cumulants for the same measure with the generating series of planar fully simple maps. The formulas for discs and cylinders~\eqref{01rel}-\eqref{02rel} recover the $R$-transform machinery for $n=1,2$. In particular, the formula for cylinders gives an intrinsic, geometric meaning to the functional relation for $n=2$.

\medskip

The present work allows to recursively compute higher order free cumulants in certain unitary invariant matrix models, namely for measures of the form $\mathcal{Z}_N^{-1}\dd M\,e^{-N{\rm Tr}\,V(M)}$. We expect that this approach could \textit{in fine} extend the $R$-transform machinery for any order $n$ for those measures. For instance, for $n=3$, we have already established the desired functional relation between ordinary and fully simple pairs of pants, that is of topology $(0,3)$:
\begin{cor} Let $\omega_{0,2}(\theta_1,\theta_2)=\chi_{0,2}(\theta_1,\theta_2)$ be the standard bidifferential and set $\alpha=1$. Then, we have the following relation of ordinary and fully simple pairs of pants:
\begin{align}\label{03formula}
& \omega_{0,3}(\theta_1,\theta_2,\theta_3) + \chi_{0,3}(\theta_1,\theta_2,\theta_3) = \Res_{z = \theta_1,\theta_2,\theta_3} \frac{\omega_{0,2}(\theta,\theta_1)\omega_{0,2}(\theta,\theta_2)\omega_{0,2}(\theta,\theta_3)}{\dd \mathsf{x}(\theta)\dd \mathsf{y}(\theta)}  \\
& =  \dd_{1}\Big[\frac{\omega_{0,2}(\theta_1,\theta_2)\omega_{0,2}(\theta_1,\theta_3)}{\dd \mathsf{x}(\theta_1)\dd \mathsf{y}(\theta_1)}\Big] + \dd_{2}\Big[\frac{\omega_{0,2}(\theta_2,\theta_1)\omega_{0,2}(\theta_2,\theta_3)}{\dd \mathsf{x}(\theta_2)\dd \mathsf{y}(\theta_2)}\Big] + \dd_3\Big[\frac{\omega_{0,2}(\theta_3,\theta_1)\omega_{0,2}(\theta_3,\theta_2)}{\dd \mathsf{x}(\theta_3)\dd \mathsf{y}(\theta_3)}\Big]. \nonumber
\end{align}
\hfill $\star$
\end{cor}
This corollary follows from Theorem~\ref{maincor} and the details exposed in \cite[Section 6]{BG-F18}. Even if free cumulants are so far only defined for $g=0$, our work suggests that there should exist a universal theory of approximate higher order free cumulants taking into account higher genus corrections. 
For a compact introduction to all the necessary objects to understand this connection to free probability precisely, the reader could consult \cite[Section 1.6]{GF18} or many other more extended sources written by experts in free probability \cite{NicaSpeicher,MingoSpeicherBook}.

\medskip

For general (formal) unitarily-invariant measures, the underlying combinatorial objects (in ordinary or fully simple flavor) are the stuffed maps introduced in \cite{Bor14}.  It was proved that stuffed maps satisfy a generalisation of the topological recursion, called blobbed topological recursion \cite{BSblob}, where the initial data of the spectral curve is enriched by symmetric holomorphic forms in $n$ variables $(\phi_{g,n})_{2g-2+n>0}$. In \cite{BG-F18} it was conjectured that after the same symplectic exchange transformation, and a transformation of the blobs still to be described, blobbed topological recursion will enumerate fully simple stuffed maps. This conjecture already follows for the base topologies $(0,1)$ and $(0,2)$ from the formulas~\eqref{01rel}-\eqref{02rel} for discs and cylinders, since the base topologies are not altered by the blobs. It may be possible, either by studying multi-ciliated stuffed maps, or by substitution methods (at least in genus $0$), to extend the results of the present article to the case of stuffed maps. The solution of this problem would allow the compute higher order free cumulants in the full generality of \cite{speicher2007}, and progressing towards such a solution is an important motivation of the present work.

\bibliographystyle{plain}
\bibliography{Biblicombi}

\end{document}